\documentclass[leqno]{article}%
\usepackage{amsfonts,amsmath,amsthm,amssymb,times}
\usepackage{graphicx}
\usepackage{amsmath}
\usepackage{amsfonts}
\usepackage{amssymb}
\usepackage{color}
\theoremstyle{definition}

\newtheorem{example}{Example}
\newtheorem{theorem}{THEOREM}[section]
\newtheorem{lemma}{LEMMA} [section]
\newtheorem{definition}{DEFINITION}[section]

\newtheorem{remark}{REMARK}[section]

\newtheorem{proposition}{PROPOSITION}[section]

\begin{document}

%\begin{frontmatter}
\title{Nonparametric Estimation of Means on Hilbert Manifolds and Extrinsic Analysis of Mean Shapes of Contours}
%\runtitle{Estimation of Means on Hilbert Manifolds}

\author{Leif Ellingson$^\ddag$ \thanks{Research supported by National Science Foundation Grant DMS-0805977}, Vic Patrangenaru$^\dagger$ 
\thanks{Research supported by National Science Foundation Grants DMS-0805977, DMS-1106935 and the National Security Agency Grant
MSP-H98230-08-1-0058}, and Frits Ruymgaart$^\ddag$ \\
$^\dagger$Florida State University and $^\ddag$Texas Tech University}

%\runauthor{L. Ellingson et al.}

%\end{aug}

\maketitle

\begin{abstract}

{

Motivated by the problem of nonparametric inference in high level digital image analysis,
we introduce a general extrinsic approach for data analysis on Hilbert manifolds with a focus on
means of probability distributions on such sample spaces. To perform inference on these means, we appeal
to the concept of neighborhood hypotheses from functional data analysis and derive a one-sample
test. We then consider analysis of shapes of contours lying in the  plane. By embedding the corresponding sample space of
such shapes, which is a Hilbert manifold, into a space of Hilbert-Schmidt operators, we can define
extrinsic mean shapes of planar contours and their sample analogues. We apply the general methods to this problem
while considering the computational restrictions faced when utilizing digital imaging data.
Comparisons of computational cost are provided to another method for analyzing shapes of contours.

}

\vskip6pt

%{{\bf \noindent Keywords}: extrinsic means, planar contours,
% data analysis on {  Hilbert
%manifolds}, statistical shape analysis, nonparametric bootstrap,
%digital image analysis, automated randomized landmark selection\\}

\end{abstract}

%\newpage
%\Large
\section{Introduction}
{
It is the purpose of this paper to introduce general methodology for data analysis on infinite dimensional Hilbert manifolds.  Nonparametric procedures for inference on the population mean are included, focusing on an extrinsic approach.  Theoretical results for both estimation and testing hypotheses are derived.  Dette and Munk (1998) \cite{Munk:98} rekindled the interest in neighborhood hypotheses by showing that they are appropriate and useful in a nonparametric functional context.  Since ``in practice" once cannot expect an infinite dimensional object to be exactly equal to a prescribed hypothesized object, the neighborhood hypothesis will be employed too in this paper.

Just as in finite dimensions, the general theory could be applied to a variety of special manifolds.  However, here we will restrict ourselves to the important case of projective spaces that will be embedded in the Hilbert-Schmidt operators, which is useful in the analysis of shapes of infinite dimensional configurations lying in the plane.  Necessary computational considerations are considered for the implementation of the methodology and the computational cost turns out to compare favorably with that of other procedures used in shape analysis.

Let us now turn to a brief discussion of the existing literature that is most relevant for the present paper.  To the best of our knowledge, the combination of the theory for functional data in infinite dimensional linear spaces with the geometry for infinite dimensional manifolds for the purpose of  nonparameteric analysis is new.  There are, however, some procedures that seem to lack the underlying asymptotic theory needed for a nonparametric analysis.  These will be reviewed along with some pertinent theoretical results for functional data, and some relevant existing methods for finite dimensional manifolds.

A number of statistical methodologies have been developed for the analysis of
data lying on Hilbert spaces for the purpose of studying functional data. Some multivariate methods, such as PCA, have useful extensions in functional data analysis   (Lo$\grave{e}$ve \cite{Lo:1977}). For dense functional data, the asymptotics of the resulting eigenvalues and eigenvectors were studied by Dauxois et. al. (1982) \cite{DaPoRo:1982}. Even for sparse functional data, these methods have been proved useful (see Hall et. al. (2006)\cite{haMuWa:2006}, M$\ddot{u}$ller et. al.(2006)\cite{Mueler:2006}) and have multiple applications.
There are, nevertheless, techniques defined for multivariate analysis that often fail to be directly generalizable to infinite-dimensional data, especially when such data is nonlinear.
New methodologies have been developed to account for these high dimensional problems, with many of these presented in a standard text by
Ramsay and Silverman (2005) and given by references therein. For high dimensional inference on Hilbert spaces,
Munk and Dette (1998) \cite{Munk:98} utilized the concept of neighborhood hypotheses
for performing tests on nonparametric regression models.  Following from this approach,
Munk {\it et al.} (2008) \cite{MuPaPaPaRu:2008} developed one-sample and multi-sample
tests for population means.

However these methods do not account for estimation of means on infinitely dimensional curved spaces, such as Hilbert manifolds. In order to properly analyze such data, these methods must
be further generalized and modified. A key example in which such data arises is in the
statistical analysis of direct similarity shapes of planar contours, which can be viewed as outlines
of 2D objects in an image.

Unlike functional data analysis, which started from dense functional data and was extended to sparse functional data, the study of shapes in the plane originated with
D. G. Kendall (1984) \cite{Kend:1984}, which showed that the space $\Sigma_2^k$ of direct
similarity shapes of nontrivial planar finite configurations of $k$ points is a complex projective space $\mathbb CP^{k-2}.$ However, definitions of location
and variability parameters for probability distributions were considered much later.
While the full Procrustes estimate of a mean shape was defined by Kent (1992),
a nonparametric definition of a mean shape was first introduced by Ziezold (1994) \cite{Zi:1994} based upon the notion of Fr\'echet population mean (Fr\'echet (1948) \cite{Fr:1948}, Ziezold (1977) \cite{Zi:1977}).  This approach was followed by Ziezold
(1998) \cite{Zi:2000}, Le and Kume (2000) \cite{LeKu:2000}, Kume and Le (2000
\cite{KuLe:2000}, 2003 \cite{KuLe:2003}), Le (2001) \cite{Le:2001}, Bhattacharya
and Patrangnearu (2003) \cite{BhPa:2003}, and Huckemann and Ziezold (2006) \cite{HuZi:2006}.

While a majority of these methods have defined means in terms of Riemannian distances,
as initially suggested by Patrangenaru (1998) \cite{Pa:1998}, a Veronese-Whitney (VW) extrinsic mean similarity shape was also introduced by Patrangenaru (op.cit.), in terms of the VW embedding of $\mathbb CP^{k-2}$
into the space $S(k-1,\mathbb C)$ of selfadjoint $(k-1)\times (k-1)$ matrices.
%Bhattacharya and Patrangenaru (2005) \cite{BhPa:2005} noted that this extrinsic
%mean is the Fr\'echet mean using the distance on $\mathbb CP^{k-2}$ induced
%by the Euclidean distance on $S(k-1,\mathbb C)$.
The asymptotic distribution of
this extrinsic sample mean shape and the resulting bootstrap distribution are
given in Bhattacharya and Patrangenaru (2005) \cite{BaPa:2005}, Bandulasiri et.
al (2009) \cite{BaBhPa:2009} and Amaral et at. (2010) \cite{AmDrPaWo:2010}.
}

Motivated by {the pioneering work of Zahn and Roskies
(1972) and by other} applications in object recognition from digital
images, Grenander (1993) \cite{Gr:1993} considered shapes as points
on some infinite dimensional {space}. A manifold model
for direct similarity shapes of planar closed curves, first
suggested by Azencott (1994) \cite{Az:1994}, was pursued in Azencott
et. al. (1996)\cite{AzCoYo:1996}, and detailed by Younes (1998
\cite{Yo:1998}, 1999 \cite{Yo:1999}). This idea gained ground at the
turn of the millennium, with more researchers studying shapes of
planar closed curves (eg. Sebastian et. al. (2003) \cite{SeKlKi:2003}).

{
Klassen et al. (2004) \cite{KlSrMiJo:2004}, Michor and Mumford (2004)
\cite{MicMum:2004} and Younes et al. (2008) \cite{YoMiShMu:2008}
follow the methods of Small (1996) \cite{Small:1996} and Kendall by
defining a Riemannian structure on a shape manifold. Klassen { et al.}
(op. cit) compute an intrinsic sample mean shape, which is a Fr\'echet
sample mean for the chosen Riemannian distance. The related papers Mio
and Srivastava (2004) \cite{MiSr:2004}, Mio et al. (2007) \cite{MiSrJo:2005},
Srivastava et al. (2005) \cite{SrJoMiLi:2005}, Mio et al. (2005) \cite{MiSrKl:2005},
Kaziska and Srivastava (2007) \cite{KaSr:2007}, and Joshi {it et al.} (2007)
\cite{JoSrKlJe} computean intrinsic sample mean similarity shape of closed curves modulo reparameterizations, for a
Riemannian metric of their preference, out of infinitely many Riemannian metrics
available. To compute these types of means, those papers use iterative gradient
search algorithms, as closed-form solutions for the means do not exist.

%This is problematic for a number of reasons. First,

However, these approaches only consider computational aspects without addressing
fundamental definitions of populations of shapes, population means and
population covariance operators, thus making no distinction between a population
parameter and its sample estimators. As such, the idea of statistical inference
is absent from these papers. 

This paper is organized as follows. In Section 2, we introduce methodology for data analysis on Hilbert manifolds, including a definition for the extrinsic mean (set) of a random object an embedded Hilbert manifold and introduce an inference procedure for
a neighborhood hypothesis for an extrinsic mean. In Section 3, we define the
space of direct similarity shapes of planar contours.  Section
4 is dedicated to the derivation of the asymptotic distribution of the
extrinsic sample mean contour shape. Due to the infinite-dimensionality, the
sample mean cannot be properly studentized, so in Section 5, we instead apply
the neighborhood hypothesis test to this problem. 

The remainder of the paper concerns application of the methodology to digital imaging data. In section 6 we address the representation,  approximation,
and correspondence problems faced when working with such data in practice. In Section 7, we present examples of the neighborhood hypothesis test using
contours extracted from a database of silhouettes from digital images, collected by Ben Kimia \cite{Kimia}. In Section 8, we use nonpivotal nonparametric bootstrapping (see Efron (1979) \cite{Efron:79}, Hall (1992) \cite{Hall:1997}) to form confidence regions for the
extrinsic mean shape for examples from the same database and compare
computational speed of our method to that of Joshi et. al. (2007) \cite{JoSrKlJe}.
}
The paper ends with a discussion suggesting an extension of other methodologies from functional data or for data on finite dimensional manifolds to data analysis on Hilbert manifolds. An extension of the shape analysis methods to  more complicated, infinite
dimensional features captured in digital images, such as edge maps obtained
from gray-level images, is also suggested.
\section{Inference for means on Hilbert manifolds}
In this section we assume that $\mathbf H$ is a {\it separable, infinite dimensional
Hilbert space} over the reals. Any such space is isometric with
$l_2,$ the space of sequences $x = (x_n)_{n \in \mathbb N}$ of reals
for which the series $\sum_{n=0}^\infty x_n^2$ is convergent, with
the scalar product $<x,y> = \sum_{n=0}^\infty x_ny_n.$ A Hilbert
space with the norm $\|v\| = \sqrt{<v, v>},$ induced by the scalar
product, becomes a Banach space. Differentiability can be defined
with respect to this norm.
\begin{definition} \label{differentiable-banach} A function $f$ defined on an
open set $U$ of a Hilbert space $\mathbf H$ is Fr\'echet differentiable at a
point $x \in \ U,$ if there is a linear operator $T : \mathbf H \to \mathbf H,$
such that if we set
\begin{equation}\label{banach-dif}
\omega_x(h) = f(x + h ) - f ( x ) - T(h),
\end{equation}
then
\begin{equation}\label{banach-dif2}
\lim_{ h \to 0} \frac{\|\omega_x(h)\|}{\|h\|} = 0.
\end{equation}
\end{definition}
Since $T$ in Definition \ref{differentiable-banach} is unique, it is
called the differential of $f$ at $x$ and is also denoted by $d_xf.$
\begin{definition} \label{hilbert-manifold} A chart on a separable metric space
$(\mathcal M, \rho)$ is a one to one homeomorphism $\varphi: U \to
\varphi(U)$ defined on an open subset $U$ of $\mathcal M$ to a
Hilbert space $\mathbf H.$ A Hilbert manifold is a separable metric
space $\mathcal M,$ that admits an open covering by domain of
charts, such that the transition maps $\varphi_V \circ
\varphi_U^{-1} : \varphi_U(U \cap V) \to \varphi_V(U \cap V)$ are
differentiable.
\end{definition}
\begin{example}\label{proj-space}
The projective space $P(\mathbf H)$ of a Hilbert space $\mathbf H,$
the space of all one dimensional linear subspaces of $\mathbf H,$ has a
natural structure of Hilbert manifold modelled over $\mathbf H.$
Define the distance between two vector lines as their {angle}, and,
given a line $\mathbb{L} \subset \mathbf H,$  a neighborhood
$U_\mathbb{L}$ of $\mathbb{L}$ can be mapped via a homeomorphism
$\varphi_{\mathbb{L}}$ onto an open neighborhood of the
orthocomplement $\mathbb L^\bot $ by using the decomposition
$\mathbf H = \mathbb L \oplus \mathbb L^\bot.$ Then for two
perpendicular lines $\mathbb{L}_1$ and $\mathbb{L}_2,$ {
it is easy to } show that the transition maps
$\varphi_{\mathbb{L}_1}\circ\varphi_{\mathbb{L}_2}^{-1}$ are
differentiable as maps between open subsets in $\mathbb L_1^\bot ,$
respectively in $\mathbb L_2^\bot.$ A countable orthobasis of
$\mathbf H$ and the lines $\mathbb{L}_n, n \in \mathbb N$ generated
by the vectors in this orthobasis is used to cover $P(\mathbf H)$ with the
open sets $U_{\mathbb{L}_n}, n \in \mathbb N. $ Finally, use the fact
that for any line $\mathbb L, \mathbb L^\bot $ and $\mathbf H$ are
isometric as Hilbert
spaces. The line $\mathbb{L}$ spanned by a nonzero vector $\gamma \in \mathbf H$ is
usually denoted $[\gamma]$ when regarded as a projective point on
$P(\mathbf H)$ and will be denoted as such henceforth.
\end{example}

Similarly, one may consider complex Hilbert manifolds,
modeled on Hilbert spaces over $\mathbb C.$ A vector space over
$\mathbb C$ can be regarded as a vector space over the reals, by
restricting the scalars to $\mathbb R;$ therefore any complex
Hilbert manifold automatically inherits a structure of real Hilbert
manifold.

\subsection{Extrinsic analysis of means on Hilbert manifolds}
For Hilbert spaces that do not have a linear structure, standard methods
for data analysis on Hilbert spaces cannot directly be applied.  To
account for this nonlinearity, one may instead perform extrinsic analysis
by embedding this manifold in a Hilbert space.
\begin{definition} An embedding of a Hilbert manifold $\mathcal M$
in a Hilbert space $\mathbb H$ is a one-to-one differentiable
function $j: \mathcal M \to \mathbb H,$ such that for each $x \in
\mathcal M,$ the differential $d_x j$ is one to one, and the range
$j(\mathcal M)$ is a closed subset of $\mathbb H$ and the topology
of $\mathcal M$ is induced via $j$ by the topology of $\mathbb H.$
\end{definition}
\begin{example}\label{hilbert-schmidt}
{   We embed $P({\bf H})$ in $\mathcal{L}_{HS}=\bf{H}
\otimes \bf{H},$ the space of Hilbert-Schmidt operators of ${\bf H}$
into itself, via the Veronese-Whitney (VW) embedding $j$ given by
\begin{equation}\label{veronese1}
j([\gamma]) = \frac{1}{\|\gamma\|^2}\gamma \otimes \gamma.
\end{equation}
If $\|\gamma\|=1$, this definition can be reformulated as
\begin{equation}\label{veronese-hilbert2}
j([\gamma]) = \gamma \otimes \gamma.
\end{equation}
%Equation
%\eqref{veronese-hilbert2} is equivalent to
%\begin{equation}\label{veronese3}
%j([\gamma])(\beta) = <\beta,\gamma > \gamma.
%\end{equation}
The range of this embedding is the submanifold $\mathcal M_1$ of
rank one Hilbert-Schmidt operators of ${\bf H}.$ }
\end{example}
{  To define a location parameter for probability distributions on a Hilbert manifold,}
the concept of extrinsic means from Bhattacharya and Patrangenaru (2003, 2005)) is extended
below to the infinite dimensional case.
\begin{definition} \label{def:ext-mean}  If $j:\mathcal M \to \mathbb H$ is an embedding of a
Hilbert manifold in a Hilbert space, the chord distance $\rho$ on
$\mathcal M$ is given by $\rho(x,y) = \|j(x) - j(y)\|,$ and given a
random object $X$ on $\mathcal M,$ the associated Fr\'echet function
is
\begin{equation}\label{frechet}
\mathcal F_j(x) = E(\|j(X) - j(x)\|^2).
\end{equation}
The set of all minimizers of $\mathcal F_j$ is called the extrinsic
mean set of $X.$ If the extrinsic mean set has one element only,
that element is called the extrinsic mean and is labeled $\mu_{E,j}$
or simply $\mu_E.$
\end{definition}
\begin{proposition} \label{ext} Consider a random object $X$ on $\mathcal M$
that has an extrinsic mean set. Then (i) $j(X)$ has a mean vector
$\mu$ and (ii) the extrinsic mean set is the set of all points $x
\in \mathcal M,$ such that $j(x)$ is at minimum distance from $\mu.$
(iii) In particular, $\mu_E$ exists if {  there is a unique point on $j(\mathbf M)$ at minimum distance from $\mu,$ the projection $P_j(\mu)$ of $\mu$ on $j(\mathbf M),$ and in this case $\mu_E =
j^{-1}(P_j(\mu)).$}
\end{proposition}
{\bf Proof.} Let $Y = j(X).$ Note that the Hilbert space $\mathbb H $ is
complete as a metric space, therefore $\inf_{y \in \mathbb H }E(\|Y - y\|^2)
=  \min_{y \in \mathbb  H}E(\|Y - y\|^2) \leq \min_{y \in j(\mathcal M)}E(\|Y -
y\|^2),$ and from our assumption, it follows that $\inf_{y \in \mathbb
H}E(\|Y - y\|^2)=  \min_{y \in \mathbb H }E(\|Y - y\|^2)$ is finite, which
proves (i). To prove (ii), assume for $\nu$ is a point in the
extrinsic mean set, and $ x $ is an arbitrary point on $\in \mathcal
M.$ From $ E(\|j(\nu) - Y\|^2)\leq E(\|j(x) - Y\|^2)\,$ and, since
$j(\nu) - \mu$ and $j(x) - \mu$ are constant vectors, it follows
that
\begin{equation}\label{proj1} \|j(\nu) - \mu\|^2 \leq \|j(x) -
\mu\|^2 + 2E(<j(x)-j(\nu), \mu - Y>). \end{equation}
It is obvious
that the expected value on the extreme righthand side of equation
\eqref{proj1} is zero \qedhere
We now consider a property that is critical for having
a well-defined, unique extrinsic mean.
\begin{definition} A random object $X$ on a Hilbert manifold
$\mathcal M$ embedded in a Hilbert space is $j$-nonfocal if there is
a unique point $p$ on $j(\mathcal M)$ at minimum distance from
$E(j(X)).$
\end{definition}
\begin{example} The unit sphere
$S(\mathbb H) = \{ x \in \mathbb H, \|x\| = 1\}$ is a Hilbert manifold
embedded in $\mathbb H$ via the
inclusion map $j: S(\mathbb H) \to \mathbb H, j(x) = x.$ A random object
$X$ on $S(\mathbb H)$ of mean $\mu$ is $j$-nonfocal, if $\mu \neq 0.$
\end{example}
Using this property, one may give an explicit formula for the extrinsic mean for a random object
on $P(\mathbf H)$ with respect to the VW embedding, which we will call VW mean.
\begin{proposition}\label{ext} Assume $X=[\Gamma]$ is a random object
in $P({\bf H}).$ Then the
VW mean of $X$ exists if and only if
$E(\frac{1}{\|\Gamma\|^2}\Gamma \otimes \Gamma)$ has a simple
largest eigenvalue, in which case, the distribution is $j$-nonfocal.
 In this case the VW mean is $\mu_E =
[\gamma],$ where $\gamma$ is an eigenvector for this eigenvalue.
\end{proposition}
{\bf Proof.}
 We select an arbitrary point $[\gamma] \in P(\mathbf H),
\|\gamma\| = 1.$ The spectral decomposition of $\Lambda =
E(\frac{1}{\|\Gamma\|}\Gamma)$ is $\Lambda = \sum_{k=1}^\infty
\delta^2_kE_k, \delta_1 \geq \delta_2 \geq \dots $ where for all $k
\geq 1, E_k = e_k\otimes e_k, \|e_k\| = 1,$ therefore if $\gamma =
\sum_{k=1}^\infty x_k e_k, \sum_{k=1}^\infty x_k^2 < \infty,$ then
$\|j(\gamma) - \mu \|^2 = \|\gamma\otimes \gamma\|^2 +
\sum_{k=1}^\infty \delta_k^2 - 2<\Lambda, \gamma\otimes\gamma>.$ To
minimize this distance it suffices to maximize the projection of the
unit vector $\gamma\otimes\gamma$ on $\Lambda.$ If $\delta_1 =
\delta_2$ there are the vectors $\gamma_1 = e_1$ and $\gamma_2 =
e_2$ are both maximizing this projection, therefore there is a
unique point $j([\gamma])$ at minimum distance from $\Lambda$ if and
only if $\delta_1 > \delta_2.$ \qedhere
{  A definition of a covariance parameter is also needed in order
to define asymptotics and perform inference on an extrinsic mean.}
The following result is a straightforward extension of the
corresponding finite dimensional result in Bhattacharya and
Patrangenaru (2005). The tangential component $tan(v)$ of $ v \in \mathbb{H}$
 w.r.t. the orthobasis $e_a (P_j
(\mu))\in T_{P_j (\mu)}j(M), a = 1, 2,\cdots, \infty$ is given by
\begin{equation}\label{e:tancomp}
tan(v) = \sum_{a=1}^\infty (e_a (P_j (\mu))\cdot v)e_a (P_j (\mu)).
\end{equation}
Then given the $j$-nonfocal random object
$X,$ extrinsic mean $\mu_E$, and covariance operator of $\tilde \Sigma = cov(j(X)),$ if $f_a(\mu_E) = d_{\mu_E}^{-1}(e_a (P_j (\mu))), \forall a = 1, 2, \dots,$ then $X$
has {\em extrinsic
covariance operator} represented w.r.t. the basis $f_1(\mu_E),\cdots$ by the infinite matrix $\Sigma_{j,E}$:
\begin{eqnarray}
\Sigma_{j,E} =
\nonumber \\
\left[ \sum d_\mu P_j (e_b) \cdot e_a (P_j(\mu)) \right]_{a = 1,...}
\tilde \Sigma
 \left[ \sum d_\mu P_j (e_b) \cdot e_a(P_j
(\mu))\right]_{a = 1,...}^T .
\end{eqnarray}
 With extrinsic parameters of location and covariance now defined,
the asymptotic distribution of the extrinsic mean can be shown as in the final dimensional case (see Bhattacharya and Patrangenaru (2005)\cite{BhPa:2005}):
\begin{proposition}\label{p:vw-mean} Assume $X_1, \dots, X_n$ are i.i.d.
random objects (r.o.'s) from a $j$-nonfocal distribution on a Hilbert manifold
$\mathcal M,$ for a given embedding $j: \mathcal M \to {\bf H}$ in
a Hilbert space ${\bf H}$ with extrinsic mean $\mu_E$ and extrinsic
covariance operator $\Sigma.$ Then, with probability one, for $n$
large enough, the extrinsic sample mean $\bar X_{n,E}$ is well
defined. If we decompose $j(\bar X_{n,E}) - j(\mu_E)$ with
respect to the scalar product into a tangential component in
$T_{j(\mu_E)}j(\mathcal M)$ and a normal component
$N_{j(\mu_E)}j(\mathcal M),$ then
\begin{equation}\label{eq:clt}
\sqrt{n}(tan(j(\bar X_{n,E}) - j(\mu_E))) \to_d \mathcal G,
\end{equation}
where $\mathcal G$ has a Gaussian distribution in
$T_{j(\mu_E)}j(\mathcal M)$ with extrinsic covariance operator
$\Sigma_{j,E}.$
\end{proposition}
\subsection{A one-sample test of the neighborhood hypothesis}
{Following from a neighborhood method in the context of
regression by Munk and Dette (1998)\cite{Munk:98}, Munk { et al.}
(2008)$\cite{MuPaPaPaRu:2008}$ developed tests for means of random
objects on Hilbert spaces. We now adapt this methodology for tests
for extrinsic means.  First, however,}
we will need the following useful extension of
Cramer's delta method. {  The proof is left to the reader.}
\begin{theorem}\label{t:Cramer-ext} For $ a =1, 2$ consider an embedding
$j_a: \mathcal M_a \to {\bf H}_a$ of a Hilbert manifold
$\mathcal M_a$ in
a Hilbert space ${\bf H}_a.$
Assume $X_1, \dots, X_n$ are i.i.d.
r.o.'s from a $j_1$-nonfocal distribution on
$\mathcal M_1$ for a given embedding $j_1: \mathcal M_1 \to {\bf H}_1$ in
a Hilbert space ${\bf H}_1,$ with extrinsic mean $\mu_{E}$ and extrinsic
covariance operator $\Sigma.$ Let $\varphi: \mathcal M_1 \to \mathcal M_2
$ be a differentiable function, such that $\varphi(X_1)$ is a $j_2$-nonfocal
r.o. on $\mathcal M_2.$ Then
\begin{equation}\label{eq:cramer}
\sqrt{n}tan_{j_2(\varphi(\mu_E))}(j_2(\varphi(\bar X_{n,E})) - j_2(\varphi(\mu_E))) \to_d Y,
\end{equation}
where $Y \sim \mathcal N(0, d_{\mu_E}\varphi^{*}\Sigma
d_{\mu_E}\varphi).$ Here $L^{*}$ is the adjoint operator of $L.$
\end{theorem}
{
We can now define the neighborhood hypothesis procedure for tests of extrinsic means. Assume $\Sigma_j$ is the extrinsic covariance operator of a random object $X$ on the Hilbert manifold $\mathcal M,$ with respect to the embedding $j:\mathcal M \to \mathbb H.$ Let $\mathbf M_0$ be a compact submanifold of $\mathcal M.$  Let $\varphi_0: \mathcal M \to \mathbb R$ be the function
\begin{equation}\label{eq:phi0}
\varphi_0(p) = \min_{p_0 \in \mathbf M_0}\|j(p) - j(p_0)\|^2,
\end{equation}
and let $\mathbf M_0^\delta, \mathbf B_0^\delta$ be given respectively by
\begin{eqnarray}\label{eq:m0delta}
\mathbf M_0^\delta = \{p \in \mathcal M, \varphi_0(p) \le \delta^2 \}, \nonumber \\
\mathbf B_0^\delta = \{p \in \mathcal M, \varphi_0(p) = \delta^2,  \}.
\end{eqnarray}
Since $\varphi_0$ is Fr\'echet differentiable and all small enough $\delta >0$ are regular values of $\varphi_0,$ it follows that $\mathbf B_0^\delta $ is a Hilbert submanifold of codimension one in $\mathcal M.$ Let $\nu_p$ be the normal space at a points $p \in \mathbf B_0^\delta,$ orthocomplement of the tangent space to $\mathbf B_0^\delta $ at $p.$ We define $\mathbf B_0^{\delta,X}$
\begin{equation}\label{eq:b-delta-x-0}
\mathbf B_0^{\delta,X} = \{p \in \mathbf B_0, \Sigma_j|_{\nu_p} \text{is positive definite}  \}.
\end{equation}
\begin{definition}\label{def:neigh} The neighborhood test consists of testing the following two hypotheses:
\begin{eqnarray}\label{eq:m0delta}
H_0 : \mu_E \in \mathbf M_0^\delta \cup \mathbf B_0^{\delta,X}, \nonumber \\
H_A : \mu_E \in (\mathbf M_0^\delta)^c \cap (\mathbf B_0^{\delta,X})^c.
\end{eqnarray}
\end{definition}
Munk et al. (2008)
$\cite{MuPaPaPaRu:2008}$ show that, in general, the test statistic
for these types of hypotheses has an {asymptotically} standard
normal distribution for large sample sizes, in the case of random objects on Hilbert spaces.
Here, we consider neighborhood hypothesis testing for the particular situation in which the
submanifold $\mathbf M_0$ consists of a point $m_0$ on $\mathcal M.$   We set
$\varphi_0 = \varphi_{m_0},$ and since $T_{m_0}\{ m_0\} = 0$ we will prove the following result.
\begin{theorem} \label{t:m0delta} If $M_0 = \{m_0\},$ the test
statistic for the hypotheses specified in (\ref{eq:m0delta}) has an asymptotically
standard normal distribution and is given by:
\begin{equation}
\label{test-stat} { T_n} = \sqrt{n} \{
\varphi_{m_0}(\hat{\mu}_E)-\delta^2 \} /s_n,
\end{equation}
where  \begin{equation} \label{s-nu} s_n^2 = 4 \langle \hat{\nu} ,
S_{E,n} \hat{\nu} \rangle\end{equation} and
\begin{eqnarray} \label{extrinsic-s-cov} S_{E,n}=\frac{1}{n} \sum_{i=1}^{n}
(\tan_{\hat{\tilde{\mu}}}d_{\overline{j(X)}_n}P_j(j(X_i) -
\overline{j(X)}_n))\otimes \nonumber \\
\otimes(\tan_{\hat{\tilde{\mu}}}d_{\overline{j(X)}_n}P_j(j(X_i) -
\overline{j(X)}_n))\end{eqnarray}
is the extrinsic sample covariance operator for $\{X_i\}^n_{i=1}$, and
\begin{equation}\label{hat-nu}\hat{\nu} =
(d_{\hat{\mu}_{E,n}}j)^{-1}\widehat{tan}_{j(\hat{\mu}_{E,n})}
(j(m_0)-j(\hat{\mu}_{E,n})).
\end{equation}
\end{theorem}
{\bf Proof.}  The function $\varphi_0$ given in equation \eqref{eq:phi0} defined on $\mathcal M$ can be written as a composite function $\varphi_0 = \Phi_A \circ j,$ where $\Phi_A(x) = \|x - A\|^2$ is differentiable on
 $\mathbb H \backslash \{A\},$ with the differential at $x$
 given by
$ d_x\Phi_A(y) = 2< y, x - A>. $
 Since $j(\mathcal M)$ is a
 submanifold of $\mathbb H,$ the restriction $\phi_A$ of $\Phi_A(x)$ to
 $j(\mathcal M)$ is a differentiable function, with the differential
\begin{equation}\label{ed:differential-sq-dist}
 d_p\phi_A(y) = 2< y, p - A>, \forall y \in T_p j(\mathcal M).
 \end{equation}
Note that $\varphi_{m_0}(p) = \phi_{j({m_0})}(j(p)),$ therefore,
given that the differential $d_{p}j$ is a vector space
isomorphism, we obtain
\begin{equation}\label{ed:diff-ext-sq-dist}
 d_{p}\varphi_{m_0}(u) = 2< d_{p}j(u), j(p) - j(m_0)>, \forall u \in T_{p}
\mathcal M,
 \end{equation}
 and in particular
\begin{equation}\label{ed:diff-ext-sq-dist-at-mu-e}
 d_{\mu_E}\varphi_{m_0}(u) = 2< d_{\mu_E}j(u), j(\mu_E) - j(m_0)>, \forall u
 \in T_{\mu_E}
\mathcal M,
 \end{equation}
 that is
 \begin{equation}\label{ed:diff-ext-sq-dist-at-mu-e-tan}
 d_{\mu_E}\varphi_{m_0} = 2 d_{\mu_E}j \otimes tan ( j(\mu_E) -
 j(m_0)).
 \end{equation}
 Since the null hypothesis \eqref{eq:m0delta} is accepted as long as
 $\varphi_{m_0}(\mu_E) < \delta^2,$ we derive the asymptotic distribution
 of $ \varphi_{m_0}(\hat{\mu}_E)$ under $\varphi_{m_0}(\mu_E) = \delta^2. $
From Proposition \ref{t:Cramer-ext}, it follows that
\begin{equation}\label{e:as-phim0hat}
\sqrt{n}(\varphi_{m_0}(\hat{\mu}_E)- \varphi_{m_0}({\mu}_E)) \to_d
Y,
\end{equation}
where $Y \sim \mathcal N(0,(d_{\mu_E}\varphi_{m_0})^*\Sigma_E
d_{\mu_E}\varphi_{m_0}),$ we see that the random variable
\begin{equation}\label{eq:z}
Z_n = \frac{\sqrt{n}(\varphi_{m_0}(\hat{\mu}_E)-
\varphi_{m_0}({\mu}_E))}{\sqrt{(d_{\mu_E}\varphi_{m_0})^*\Sigma_E
d_{\mu_E}\varphi_{m_0}}} \end{equation} has asymptotically a standard
normal distribution. From equation
\eqref{ed:diff-ext-sq-dist-at-mu-e-tan}, if we set
\begin{eqnarray}\label{eq:nu}
\nu = (d_{\mu_E}j)^{-1}tan ( j(\mu_E) -
 j(m_0)),\nonumber\\
 \sigma^2 = 4<\nu, \Sigma_E\nu>,
 \end{eqnarray}
 then
\begin{equation}\label{eq:z-nu}
Z_n = \frac{\sqrt{n}(\varphi_{m_0}(\hat{\mu}_E)- \delta^2)}{\sigma}
\end{equation}
Finally we notice that $\hat{\nu}$ in \eqref{hat-nu} is a consistent
estimator of $\nu$ in \eqref{eq:nu} , therefore $s_n^2$ in equation
\eqref{s-nu} is a consistent estimator of $\sigma^2$ in equation
\eqref{eq:nu} and from Slutsky's theorem it follows that the test
statistic $T_n$ in equation \eqref{test-stat} has asymptotically a
$\mathcal N(0,1)$ distribution.
 \section{{  Similarity shape space of planar contours \label{inf-shape-space}}}
\par
{  Features extracted from digital images are represented by planar subsets of unlabeled points.
If these subsets are uncountable, the labels can be assigned in infinitely many ways. Here} we will consider only
contours, which are unlabeled boundaries of 2D topological disks in the plane. To keep
the data analysis stable, and to assign a {\em unique} labeling, we make the {\em
generic} assumption that there is a unique point
$p_0$ on such a contour at the maximum distance to its center of
mass so that the label of any other point $p$ on the contour is the
counterclockwise travel time at constant speed from $p_0$ to $p$.
As such, the total time needed to travel from $p_0$ to itself around
the contour once is the length of the contour. Therefore we consider
direct similarity shapes of nontrivial contours in the plane as described here.
A contour $\tilde \gamma$ is {  then regarded as} the range of a piecewise
differentiable function, that is parameterized by arclength, i.e.
$\gamma: [0, L] \rightarrow \mathbb C, \gamma(0) = \gamma(L)$ and
is one-to-one on $[0, L).$ Recall that the length of a piecewise
differentiable curve $\gamma : [a, b] \to \mathbb R^2$ is defined
as follows:
\begin{equation}\label{arc-length}
l(\tilde \gamma) = \int_a^{b}
\|\frac{\mathrm{d}\gamma}{\mathrm{d}t}(t) \| \mathrm{d}t,
\end{equation} and its center of mass (mean of a uniform distribution on
$\tilde \gamma$) is given by
\begin{equation}
z_{\tilde \gamma} = \frac{1}{L}\int_\gamma z \mathrm{d}s.
\end{equation}
 The contour ${\tilde
\gamma}$ is said to be regular if $\gamma$ is a simple closed curve
and there is a unique point $z_0 = argmax_{z \in {\tilde
\gamma}}\|z-z_{{\tilde \gamma}}\|.$ \par {  A direct similarity is a complex polynomial function in one variable of degree one.} Two contours $\tilde
\gamma_1, \tilde \gamma_2$ {\it have the same direct similarity
shape} if there is a direct similarity $S : \mathbb C \to \mathbb
C,$ such that $S(\tilde \gamma_1) = \tilde \gamma_2.$
{  The centered contour $\tilde \gamma_0 = \tilde \gamma -
z_{\tilde \gamma} = \{z - z_{\tilde \gamma}, z \in \tilde \gamma\} $
has the same direct similarity shape as $\tilde \gamma.$
\begin{definition}\label{center}{ Two regular
contours $\tilde \gamma_1, \tilde \gamma_2$ have the
same similarity shape if $\tilde \gamma_{2,0}=\lambda\tilde
\gamma_{1,0},$ where $\lambda$ is a nonzero complex number.}
\end{definition}
In order to construct the space of direct similarity shapes, we note the following.}
\begin{remark}\label{hilbert}{ A function $\gamma: S^1 \to \mathbb C$ is {\it centered}
if
$\int_{S^1}\gamma(z)ds = 0.$ We consider regular contours since the
complex vector space spanned by centered functions $\gamma$ yielding
regular contours $\tilde \gamma$ is a pre-Hilbert space. Henceforth,
we will be working with the closure of this space. This Hilbert
space ${\bf H}$ can and will be identified with the space of all
measurable square integrable centered functions from $S^1$ to
$\mathbb C.$ }
\end{remark}
{
Let $\Sigma_2^{reg}$ be the set of all direct similarity shapes
of regular contours, which is the same as the space of all shapes of
regular contours centered at zero.
\begin{remark} \label{r:key}
{  From Definition \ref{center} and Remark \ref{hilbert}, we associate a unique
piecewise differentiable curve $\gamma$ to a contour $\tilde{\gamma}$
by taking $\gamma (0) = z_0,$ the point at the maximum
distance to the center of $C,$ and by parameterizing
$\gamma$ using arc length in the counter clockwise direction. Therefore
$\Sigma_2^{reg}$ is a dense and open subset of $P({\bf H})$, the
projective space corresponding to the Hilbert space ${\bf H}.$
Henceforth, to simplify the notation, we will omit the symbol \
$\tilde{}$ in $\tilde{\gamma}$ and identify a regular contour with
the associated closed curve, without confusion.}
\end{remark}
}
\section{The Extrinsic Sample Mean Direct Similarity Shape and its Asymptotic Distribution}
Note, {  from Example 2, that $P({\bf
H})$ is a Hilbert manifold which is VW-embedded in the Hilbert space
$\mathcal{L}_{HS}, $ and, from Proposition \ref{p:vw-mean}, the  VW mean $\mu_E$ of a r.o. $X=[\Gamma]$ in $P({\bf
H}),$ is $[e_1],$
where $e_1$ is the eigenvector corresponding to the largest eigenvalue
of $\mu = E(\frac{1}{\|\Gamma\|^2}\Gamma \otimes \Gamma).$
\begin{proposition}\label{p:vw-sample-mean}
Given any VW-nonfocal probability measure $Q$ on $P({\bf
H}),$ then if $\gamma_1, \dots, \gamma_n$ is a random sample from
$\Gamma,$ then, for $n$ large enough, the VW sample mean $\hat{\mu}_{E,n}$
is the projective point of the eigenvector corresponding to the largest
eigenvalue of ${1 \over n} \sum_{i=1}^n \frac{1}{\| \gamma_i\|^2}
\gamma_i \otimes \gamma_i.$
\end{proposition}
}
{  We can now derive the asymptotic distribution of $\hat{\mu}_{E,n}$
based upon the general formulation specified in Proposition \ref{p:vw-mean}.
The asymptotic distribution of $\overline{j(X)}_n$ is as follows:
\begin{equation} \label{asymp-dist1}
\sqrt{n}(\overline{j(X)}_n- \mu ) \to_d \mathcal G
 \; \; \mathrm{as } \;  n \rightarrow
\infty,
\end{equation}
where $\mathcal G$ has a Gaussian
distribution $N_{\mathcal{L}_{HS}}(0,\Sigma)$ on $\mathcal{L}_{HS}$  a zero mean and covariance operator
$\Sigma$. From Proposition \ref{ext}, it follows that the projection $P_j:
\mathcal{L}_{HS} \rightarrow j(P(\mathbf{H})) \subset
\mathcal{L}_{HS}$ is given by
\begin{equation}\label{proj}
P_j(A)=\nu_A \otimes \nu_A,
\end{equation}
where $\nu_A$ is the eigenvector of norm 1 corresponding to the
largest eigenvalue of $A$, $P_j(\mu)=j(\mu_E),$ {  and $P_j(\overline{j(X)}_n)= j(\hat{\mu}_{E,n})$ }.
Applying the delta method to (\ref{asymp-dist1}) yields
\begin{equation} \label{asymp-dist2}
\sqrt{n}(j(\hat{\mu}_{E,n})-j(\mu_E)) \to_d
N_{\mathcal{L}_{HS}}(0,\mathrm{d}_{\mu}P_j \Sigma
(\mathrm{d}_{\mu}P_j)^T),
\end{equation}
as $n \rightarrow \infty,$ {  where $\mathrm{d}_{\mu}P_j$ denotes the
differential, {as in Definition \ref{differentiable-banach},} of the projection $P_j,$ evaluated at $\mu.$}
 It remains to find the expression for
$\mathrm{d}_{{\mu}}P_j$.
To determine the formula for the differential, we must consider the
equivariance of the embedding $J$. Because of this, we may assume
without loss of generality that $\mu =
diag\{{\delta}_a^2\}_{a=1,2,3,\dots}$.  As defined previously, the
largest eigenvalue of $\mu$ is a simple root of the characteristic
polynomial, with $e_1$ as the corresponding complex eigenvector of
norm 1, where $\mu_E = [{e}_1].$
{  An orthobasis for  $T_{[ e_1]}P(\mathbf{H})$
is formed by $e_a, i e_a,$ for $ a = 2, 3, \dots$, where $e_a$ is the
eigenvector over $\mathbb{R}$ that corresponds to the $a$-th eigenvalue.}
For any $\gamma$ which is orthogonal to ${e}_1$ w.r.t. the real scalar product, we
define the path $\psi_\gamma(t)=[\cos(t){e}_1 +\sin(t) \gamma]$.
Then $T_{j([{e}_1])}j(P(\mathbf{H}))$ is generated by the vectors
tangent to such paths at $t=0$.  Such vectors have the form $\gamma
\otimes {e}_1 + {e}_1 \otimes \gamma$. In particular, since the
eigenvectors of $\mu$ are orthogonal w.r.t. the complex scalar
product, we may take $\gamma = e_a,a=2,3,\dots,$ or $\gamma=i e_a,
a = 2,3,\dots$ to get an orthobasis for
$T_{j([{e}_1])}j(P(\mathbf{H})).$ Normalizing these vectors to have
unit lengths, we obtain the following orthonormal frame for
$a=2,3,\dots$:
{
\begin{align}
d_{{\mu}} j(e_a)&=2^{-1/2} (e_a \otimes {e}_{1} + {e}_{1}\otimes e_a),\\
d_{{\mu}} j(i e_a)&=i2^{-1/2} (e_a \otimes {e}_{1} + {e}_{1}\otimes
e_a),
\end{align}
}
As stated previously, since the map $j$ is equivariant, we may
assume that $\overline{j(X)}_n$ is a diagonal operator $D,$ with the
eigenvalues $\delta_1^2 > \delta_2^2 \geq ...$ In this case,
\begin{align}
d_{{\mu}_E}j(e_a)&=2^{-1/2} E_a^1 = F_a^1,\\
d_{{\mu}_E}j(i e_a)&=i2^{-1/2} E_a^1 = iF_a^1,
\end{align}
where $E_a^b$ has all entries zero except those in the positions
$(a,b)$ and $(b,a)$ that are all equal to 1.  From these
formulations and {   computations of the differential of $P_j$ in the
finite dimensional case} in Bhattacharya and Patrangenaru (2005),
it follows that $\mathrm{d}_{D} P_j(E_a^b) = 0,$ for all values
$a \leq b,$ except for $a=1 <b,$.  In this case
\begin{equation}\label{dp}
\mathrm{d}_{D} P_j(F_1^b) = \frac{1}{\delta_1^2 - \delta_b^2}F_1^b,
\mathrm{d}_{D} P_j(iF_1^b) = \frac{1}{\delta_1^2 - \delta_b^2}iF_1^b.
\end{equation}
Equation \eqref{dp} implies that the differential of the projection
$P_j$ at $\mu$ is the operator $Q_1$ given by
\begin{equation}
Q_1=\sum_{k=2}^\infty \frac{1}{ \delta^2_1 - \delta^2_k} E_k,
\end{equation}
where $\delta^2_1, \delta^2_2, \dots$ are the eigenvalues of
$E(\frac{1}{\|\Gamma\|^2}\Gamma \otimes \Gamma)$ and $E_1, E_2,
\dots$ are the corresponding eigenprojections. Also, in this
situation, $\mathcal{G}$ is a normally distributed random element in
$\mathcal{L}_{HS}.$ This results in the tangential component of {  the }
difference between the $j$ - images of the VW sample mean and of the
VW mean having an asymptotic normal distribution, albeit with
a degenerate covariance operator. From these computations, the
asymptotic distribution of this difference can be expressed more
explicitly in the following manner.
\begin{equation}\label{eq:asymp-distr-vw-sample-mean}
\sqrt{n} (tan(j(\hat{\mu}_{E,n})-j(\mu_E))) \xrightarrow{d} Q_1
\mathcal{G},
\end{equation}
{  where $tan(v)$ is the tangential component of $v \in j(P(\mathbf{H}))$
with respect to the basis $e_a(P_j(\mu)) \in T_{P_j(\mu)} j(P(\mathbf{H}))$,
for $a = 2, 3, \dots$
and is expressed as
\begin{equation} \label{tan-comp}
tan(v)=(e_2(P_j(\mu))^T v, e_3(P_j(\mu))^T v,, \dots)^T.
\end{equation}
}
However, this result cannot be used directly because $Q_1$, which is
calculated using the eigenvalues of $E(\frac{1}{\|\Gamma\|^2} \Gamma
\otimes \Gamma)$, and $\mu_E$ are unknown.  This problem is solved
by estimating $\mu_E$ by $\hat{\mu}_{E,n}$ and $Q_1$ in the
following manner.
\begin{equation}
\hat{Q}_1=\sum_{k=2}^\infty \frac{1}{ \hat{\delta}^2_1 -
\hat{\delta}^2_k}\hat{E}_k,
\end{equation}
where $\hat{\delta}_1, \hat{\delta}_2, \dots$ are the eigenvalues of {
\begin{equation}\label{eq:hat-mu}\hat \mu = {1 \over n} \sum_{i=1}^n \frac{1}{\|\gamma_i\|^2} \gamma_i \otimes
\gamma_i
\end{equation}
} and $\hat{E}_1, \hat{E}_2,\dots$ are the corresponding
eigenprojections. Using this estimation, the asymptotic distribution
is as follows:
\begin{equation}
\sqrt{n} (\widehat{tan}(j(\hat{\mu}_{E,n})-j(\mu_E)))
\overset{d}{\approx} \hat{Q}_1 \mathcal{G},
\end{equation}
where ``$\overset{d}{\approx}$" denotes approximate equality in distribution
and $\widehat{tan}$ is the tangential component relative to the
tangent space of $j(P({\bf H}))$ at $j(\hat{\mu}_{E,n})$ as in \eqref{tan-comp}, {  where
$\mu$ is replaced with $\hat \mu$ in equation \eqref{eq:hat-mu}}
Applying this result to (\ref{asymp-dist2}), we  arrive at the following.
\begin{theorem}\label{t:clt-ext-studentized}
If {  $\Gamma_1,\dots,\Gamma_n$ are i.i.d.r.o.'s from a VW-nonfocal distribution $Q$ on} $P(\mathbf{H})$
with VW extrinsic sample mean $\hat{\mu}_{E,n}$, then
\begin{equation} \label{asymp-dist3}
\sqrt{n}(j(\hat{\mu}_{E,n})-j(\mu_E)) \overset{d}{\approx}
\mathrm{d}_{\hat{\mu}_n}P_j \mathcal{G} \; \; \mathrm{as } \; n
\rightarrow \infty,
\end{equation}
where $\hat\mu_n = \overline{j(X)}_n$ is a consistent estimator of
$\mu.$
\end{theorem}
\begin{remark}\label{r:defense} It must be noted that because of the
infinite dimensionality of $\mathcal{G}$, in practice, a sample
estimate for the covariance that is of full rank cannot be found.
Because of this issue, this result cannot be properly studentized.
Rather than using a regularizarion technique for the
covariance that leads to complicated shape data computations, we
will drastically reduce the dimensionality via the use of the
neighborhood hypothesis methodology presented in Section 2.2.
This type of approach showed its efficiency in projective shape
analysis of planar curves in Munk et al. (2008).
\end{remark}
\section{  The One-Sample Neighborhood Hypothesis Test for Mean Shape}
Suppose that
$j:P({\bf H}) \to \mathcal{L}_{HS}$ is the {  VW embedding in
\eqref{veronese1}} {  and } $\delta > 0$ is a given positive number.
{  Using the notation in Section 2, we now can apply Theorem
\ref{t:m0delta} to random shapes of regular contours. }
Assume $x_r=[\gamma_r], \|\gamma_r\| = 1, r = 1,\dots,n$ is a random
sample from a {  VW}-nonfocal probability measure $Q.$ Then
equation \eqref{asymp-dist3}
shows that asymptotically the tangential component of the VW-sample
mean around the VW-population mean has a complex multivariate normal
distribution. Note that such a distribution has a Hermitian
covariance matrix (see Goodman, 1963 \cite{Go:1963}), therefore in this setting,
the extrinsic covariance operator and its sample counterpart are
infinite-dimensional Hermitian matrices. In particular, if we extend the CLT for
VW-extrinsic sample mean Kendall shapes in Bhattacharya and
Patrangenaru (2005), to the infinite dimensional case, the
$j$-extrinsic sample covariance {operator}  $S_{E,n},$ when regarded as
an infinite Hermitian complex matrix has the following entries
\begin{eqnarray}\label{ext-samp-cov} S_{E,n,ab}
= n^{-1} (\hat{\delta}^2_{1} -
\hat{\delta}^2_a)^{-1}(\hat{\delta}^2_{1} -
\hat{\delta}^2_b)^{-1}\hspace{3cm}
\\ \sum^n_{r=1} <e_a, \gamma_r><e_b, \gamma_r>^* |<e_{1},
\gamma_r>|^2, a, b = {  2, 3, \dots} \hspace{3cm} \nonumber
\end{eqnarray}
with respect to the complex orthobasis $e_2, e_3, e_4, \dots$ of
unit eigenvectors in {  the tangent space $T_{\hat{\mu}_{E,n}}P({\bf H})$}. Recall that this orthobasis corresponds via the
differential $d_{\hat{\mu}_{E,n}}$ with an orthobasis (over $\mathbb
C$ ) in the tangent space $T_{j(\hat{\mu}_{E,n})}j(P({\bf H})),$
therefore one can compute the components $\hat{\nu}^a$ of
$\hat{\nu}$ from equation \eqref{hat-nu} with respect to $e_2, e_3,
e_4, \dots,$ and derive for $s_n^2$ in \eqref{s-nu} the following
expression
\begin{equation} \label{s-nun} s_n^2 = 4 \sum_{a,b=2}^\infty
S_{E,n,ab}\hat{\nu}^a\overline{\hat{\nu}^b},\end{equation} where
$S_{E,n,ab}$ given in equation \eqref{ext-samp-cov} are regarded as
entries of a Hermitian matrix.
{  The test statistic $T_n$ in equation \eqref{test-stat} is defined on an infinite
dimensional Hilbert manifold. In the next section we will explain how to accurately compute approximations of $T_n$ based on finite dimensional polygonal approximations of the regular contours. }
\section{{  Approximation} of Planar Contours\label{sec_finite}}
{  Ideally, the shapes of planar contours could be
studied directly.  However, when performing computations, it is
necessary to approximate the contour by evaluating the function at
only a finite number of times.  If $k$ such stopping times are
selected, then the linear interpolation of the yielded stopping
points is a $k$-gon $z$, for which each stopping time is a vertex.
As with the contour, $z$ is a one-to-one piecewise differentiable
function that can be parametrized by arclength.  Let $L_k$ denote
the length of the $k$-gon. For $j=1,\dots,k$, let $z(t_j)$ denote
the $j$th ordered vertex, where $t_j \in [0, L_k)$ and
$z(t_1)=z(0)=z( L_k)$. It follows that, for $s \in (0,1)$, the
$k$-gon can be expressed as follows:
\begin{equation} \label{k-gon}
z(sL_k) =
  \begin{cases}
    (t_2 -s L_k) z(0) + sL_k z(t_2) &  0 < sL_k \leq t_2\\
    (t_j -sL_k) z(t_{j-1}) + (sL_k - t_{j-1}) z(t_j) &  t_{j-1} < sL_k \leq t_j \\
    (L_k -sL_k) z(t_{k}) + (sL_k - t_{k}) z(0) &  t_{k} < sL_k < L_k
  \end{cases}
  \end{equation}
for $j=3,\dots,k$.  As such, the space of direct similarity shapes of
non-self-intersecting regular polygons is dense in the space of direct
similarity shapes of regular contours.  Therefore, the theory and methodology
discussed in sections 3 through 5 hold for the shapes of these functions,
as well.  For the purposes of inference using the neighborhood hypothesis,
then, it suffices to use the test statistic as derived previously.
However, when considering these approximations, it is important to choose
the stopping times appropriately so that the contour is well approximated
by the polygon.  Additionally, one must take correspondence across contours
into consideration when working with a sample.  We will first
present an algorithm for choosing stopping times in such a way that
the $k$-gon well represents the contour and converges to it
accordingly.  Following that, we will address considerations for
working with samples.
\subsection{Random Selection of Stopping Times}
To obtain approximations, we propose to randomly select a large
number $k$ of stopping times $t_j$ from the uniform distribution
over $[0,L)$. By doing so, we insure, on one hand, that we
ultimately use a sufficiently large number of vertices so that the
$k$-gon well represents the contour. On the other hand, we ensure
the desired density of stopping points.
In order to maintain the within sample matching for a sample of
regular contours, we first find the point $z_0$ at the largest
distance from the center of the contour and choose that as $z(0)$.
We then randomly select $k-1$ stopping times from the uniform
distribution to form the $k$-gon, making sure to maintain the proper
ordering, preventing the $k$-gon from self-intersecting. This is
accomplished by sorting the selected stopping times in increasing
order.
It is important to choose an appropriate number of stopping times
for the given data.  The selected stopping points will be
distributed fairly uniformly around the contour for large values of
$k$, ensuring that the curve is accurately represented by the
$k$-gon.  However, choosing too many stopping times will needlessly
increase the computational cost of performing calculations.  This
will be most noticeable when utilizing bootstrap techniques to
compute confidence regions for the extrinsic mean shape.  Choosing
too few stopping times, though, while keeping computational cost
down, can be extremely detrimental as the stopping points may not be
sufficiently uniform to provide adequate coverage of the contour.
This can significantly distort the $k$-gon, as shown in
Fig.$\ref{f:rand_dog_multi}$.  In this particular instance, the
200-gon of the dog includes no information about the lower jaw of
the dog and little detail about one of the ears.
\begin{figure}[!ht]
\begin{center}
\includegraphics[scale=0.7]{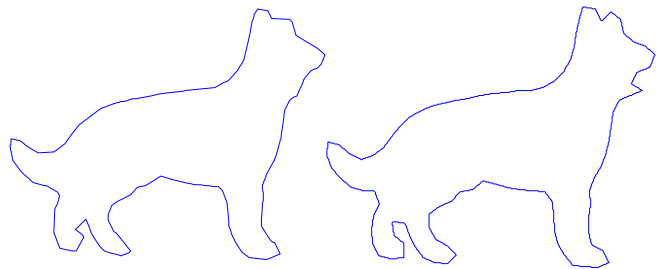}
            \caption{\scriptsize A 200-gon (left) and the associated contour (right) of a dog}
             \label{f:rand_dog_multi}
\end{center}
\end{figure}
The length of the contour $L$ can be used to assist in determining
an appropriate number of stopping points to be chosen. After
selecting an initial set of $k$ stopping times as describe above,
the length $L_k$ of the $k$-gon is
\begin{equation} \label{length-est}
L_k=\sum_{j=2}^{k+1} \|z(t_j) -z(t_{j-1})\|,
\end{equation}
where $z(t_{k+1})=z(0)$. An appropriate lower bound for the number
of stopping points can be determined by randomly selecting times for
various values of $k$. Compute $L_k$ for each of these $k$-gons
using \eqref{length-est} and compute the relative error compared to
$L$. This should be repeated many times to obtain a mean relative
error and standard deviation of the relative error for each value of
$k$ used.  To determine an appropriate number of stopping points to
use, compare the mean relative error to a desired threshold.
Additionally, the distributions of the relative errors could also be
examined.  It should be noted, however, that since digital imaging
data is discrete by nature, the contour will be represented by $K$
pixels.  As such, it is often necessary to replace $L$ by $L_K$, the
length of the closest approximation to the contour, which can be
calculated similarly to $L_k$. }
{
When using this algorithm to select stopping points, it follows that
the $k$-gon will converge to the contour.  However, when selecting
an additional stopping time $t_{j'}$, care must be taken to properly
alter $z$ in such a way that ensures that there is no
self-intersection of the $k$-gon.  To do so, simply reorder the
stopping times in increasing order and apply the resulting
permutation to the stopping points. It follows that, with
probability 1, the length of the $k$-gon between successive stopping
points will converge to 0 as the number of stopping times tends to
infinity.  This can be stated more formally as follows.
\begin{lemma}
If stopping times $s_1, s_2, \dots,s_k$ are selected from a uniform
distribution over $[0, 1)$, then
\begin{equation*}
L_{max} = \max_{j = \overline{2,k+1}} \| z( s_j L_k ) - z( s_{j-1} L_k)
\| \xrightarrow{p} 0.
\end{equation*}
\end{lemma}
\begin{proof}
\begin{align*}
P( L_{max} > \epsilon ) &= P( \mathrm{All \;} k \mathrm{ \;
stopping \; points \;are  \; within \; the\;remaining } \; L_k -\epsilon ) \\
& = P \left ( \mathrm{All \;} k \mathrm{ \; stopping \; times \;are
\; within \; the\;remaining } \; 1 - \frac{\epsilon}{L_k} \right )
\end{align*}
We can assume without loss of generality that the section of the
$k$-gon for which the distance between successive stopping times is
greater than $\epsilon^*$ is over  the interval $(0, \epsilon)$. In
addition, since the stopping times are independently chosen,
\begin{equation*}
P( L_{max} > \epsilon )=\left (F(1) - F\left(\frac{\epsilon}{L_k}
\right) \right)^k=\left (1-\frac{\epsilon}{L_k} \right)^k
\end{equation*}
where $F$ is the cdf for the uniform distribution over the interval
$[0,1)$. Taking the limit of this expression as $k \rightarrow
\infty$ results in $L_{max} \xrightarrow{p} 0$ since
\begin{equation*} \lim_{k\rightarrow \infty}
P( L_{max}> \epsilon )= \lim_{k \rightarrow \infty} \left
(1-\frac{\epsilon}{L_k} \right)^k=0
\end{equation*} \qedhere
\end{proof}
It follows immediately that the center of mass of the $k$-gon
converges to the center of mass of the contour.  While the $k$-gon
$z$ converges to the contour $\gamma$, it is also of great interest
to consider the convergence of $[z]$ to $[\gamma]$.  Since $z$ and
$\gamma$ are objects in the same space, the disparity in their
shapes can be examined by considering the squared distance $\|
j([z]) - j([\gamma])\|^2$ in $\mathcal{L}_{HS}$.  However, for the
purposes of computational comparisons, it is necessary to evaluate
the functions at $m>k$ times using \eqref{k-gon} and approximate the
distance in $S(m,\mathbb{C})$, the space of self-adjoint $m \times
m$ matrices.
\begin{figure}[!ht]
\centerline{
\includegraphics[scale=0.3]{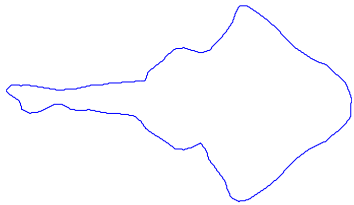}}
\caption{\scriptsize The digital image of the contour of a
stingray.} \label{f:stingray}
\end{figure}
To illustrate, consider the contour considered in Figure
\ref{f:stingray}.  The digital representation of this contour
consists of $K=764$ pixels.  Stopping times were selected using the
above algorithm to form $k$-gons for $k=3,\dots,763$.  Each $k$-gon
was then evaluated at 764 times corresponding to the each of the
pixels on the digital image of the contour.  As such, squared
distances between the $k$-gons and the contour were computed in
$S(764,\mathbb{C})$.  After this was repeated 50 times, the means
and standard deviations of the squared distances were calculated for
each value of $k$ and are shown in Figure \ref{f:rho-sq}.
\begin{figure}[!ht]
\centerline{
\includegraphics[scale=0.4]{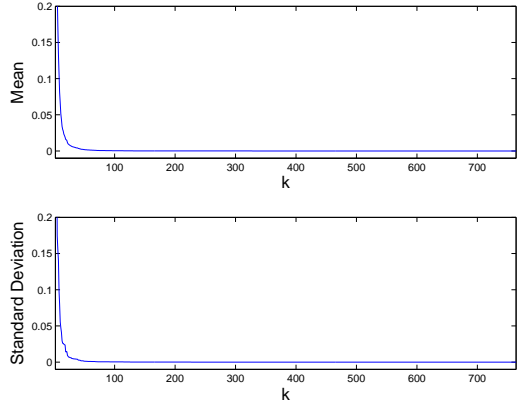}}
\caption{\scriptsize The mean and standard deviation of the squared
distance of shapes of $k$-gons to the shape of the contour, as
evaluated in $S(764,\mathbb{C})$.} \label{f:rho-sq}
\end{figure}
The mean squared distance to the contour converges quickly, showing
that the distance between $[z_k]$ and $[\gamma]$ only diminishes
slightly for $k >100$.  Moreover, the variability introduced by
selecting the stopping times randomly also rapidly approaches 0.  As
such, it is clear that $[\gamma]$ is well approximated using $k<<K$.
However, while the overall shape is well approximated, it is unclear
from this alone how well the details of $\gamma$ are approximated.
As such, using the distance between shapes may not be the best
indicator for determining a lower bound for $k$.  For this purpose,
it may be more helpful to consider $(L_K-L_k)/L_K$, the relative
error in the approximation of the length, as described previously.
For the contour in Figure \ref{f:stingray}, the relative error in
length is shown in Figure \ref{f:L-diff}.  Here, while the
variability approaches 0 quickly, the average relative error
approaches 0 at  a lower rate. As such, if it is desirable to keep
the relative error below 0.05, for this example, no fewer than 300
stopping times should be selected.
\begin{figure}[!ht]
\centerline{
\includegraphics[scale=0.4]{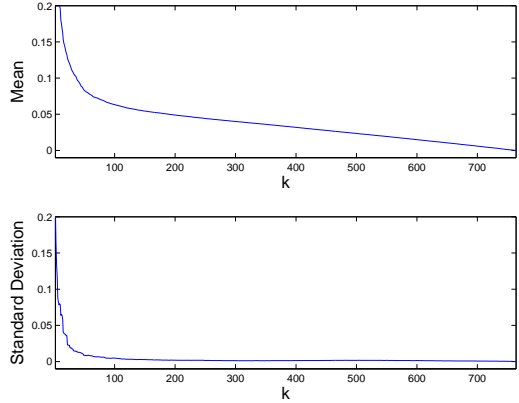}}
\caption{\scriptsize The mean and standard deviation of the relative
error of the length of the $k$-gons.} \label{f:L-diff}
\end{figure}
}
{
\subsection{Considerations for samples of contours} \label{sec-correspond}
In addition to ensuring that each contour in a sample is well
approximated, since each contour must be evaluated at $m$ times for
computations, it is necessary that each be evaluated at the same
times to maintain correspondence across the $n$ observations.  In
the ideal scenario, if each contour is well approximated by a
$k$-gon, then select $k$ stopping times $s_1,\dots,s_k\in [0,1)$
using the algorithm as described above.  For $j=1,\dots,n$, the
stopping points for $z_j$ can then be obtained by evaluating
$\gamma_j$ at times $s_1\cdot L^j, \dots, s_k \cdot L^j$, where
$L^j$ denotes the length of of $\gamma_j$.  Fig. \ref{f:corr_1}
shows two examples of utilizing this procedure for samples of
contours of hand gestures. Using the same stopping times for each
observation, 6 stopping points are highlighted in red to illustrate
the correspondence across the sample.
\begin{figure}[!ht]
\centerline{
\includegraphics[scale=0.5]{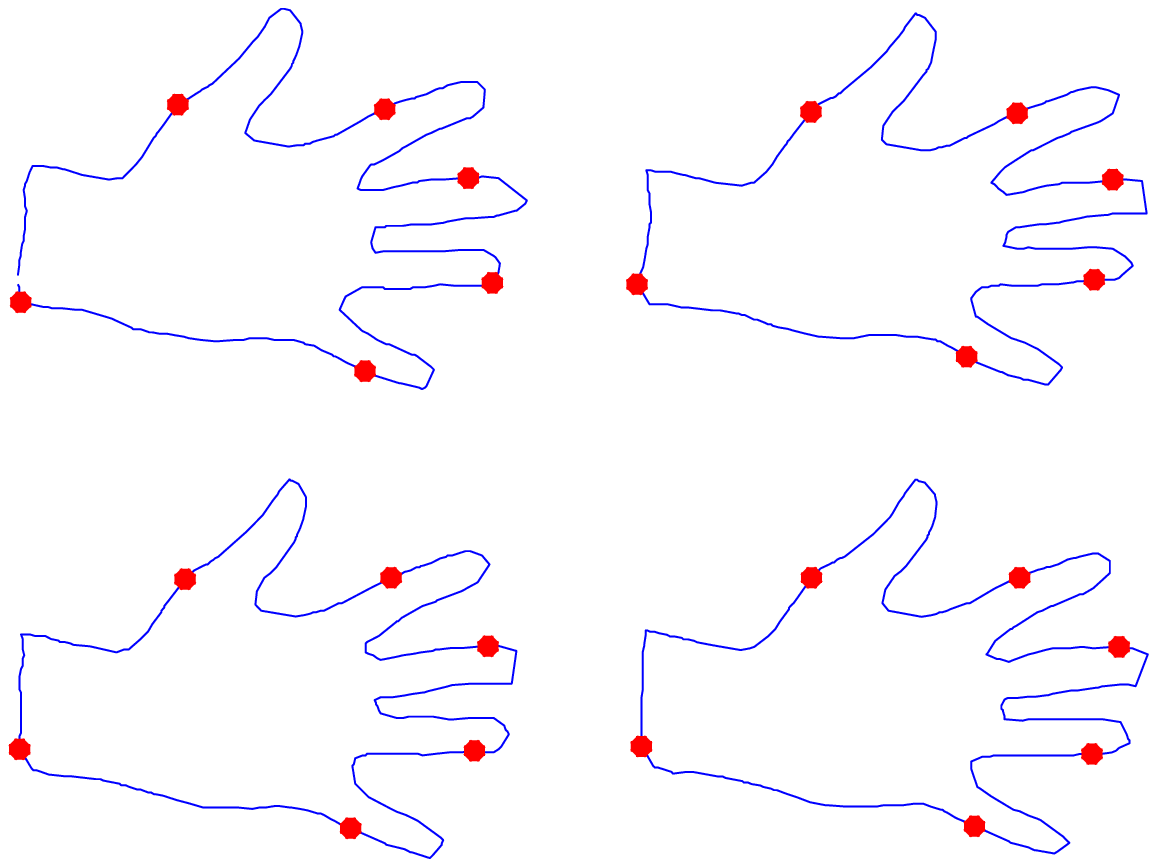}
\includegraphics[scale=0.5]{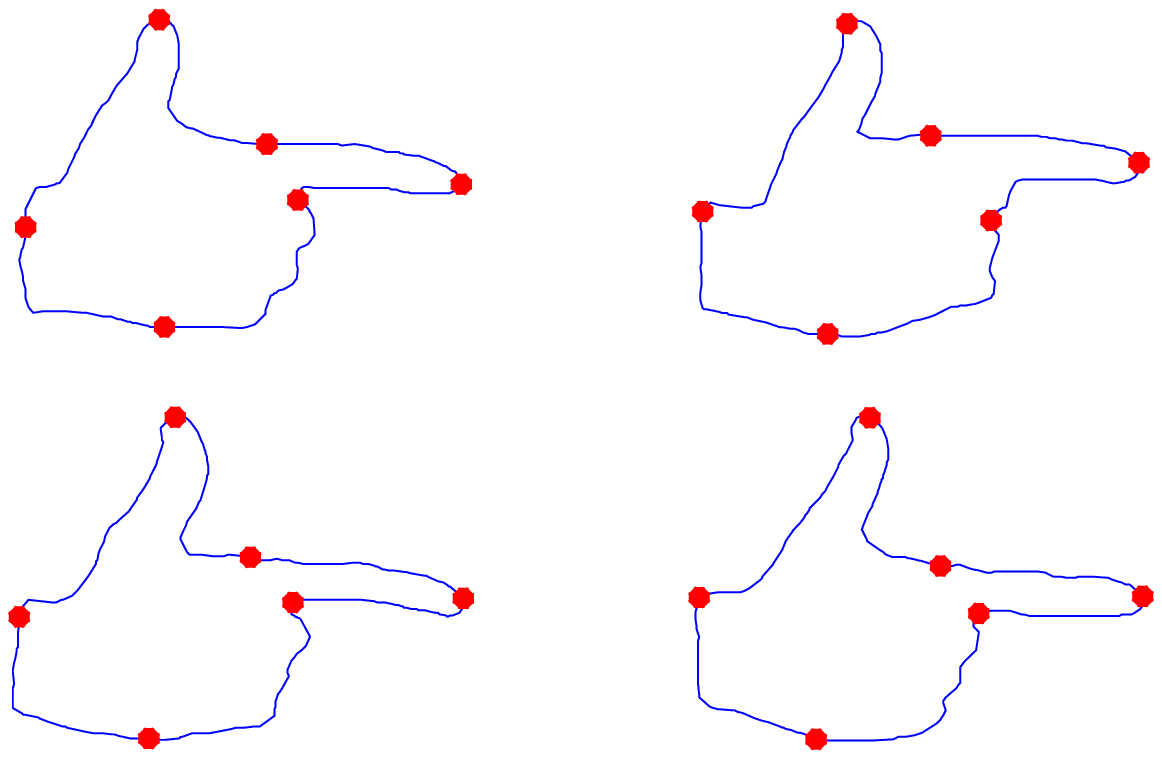}}
\caption{\scriptsize Correspondence of 6 stopping points for
contours of hand gestures of (a) the number '5'  and (b) the letter
'L'} \label{f:corr_1}
\end{figure}
Alternatively, if contour $j$ requires $k_j$ stopping points for
adequate approximation, where $k_i \neq k_j$ for at least one pair
$i,j$, then select stopping times for each contour.  Let
$\mathcal{T}_j$ denote the set of stopping times that generate the
$k_j$-gon $z_j$. In order to maintain correspondence, evaluate $z_j$
at the $m$ times contained in $\cup_{i=1}^n \mathcal{T}_i$ for
$j=1,\dots,n$.  This approach may also be utilized if each contour
is approximated using $k$ stopping points, but at different times.
Finally, even if the conditions of either of the previous scenarios
are met, it may be desired to consistently work within the same
shape space when working with multiple samples, so it may be
preferred to instead first consider approximating the contours and
then approximating each at $m$ subsequently chosen times, thus
separating the issues of approximation and correspondence. However,
for each of these scenarios, the selection of stopping points,
evaluation of the $k$-gon at $m$ times, and subsequent analysis can
be either semi-automated or fully automated, allowing for efficient
execution of the methodology.}
{
\subsection{Approximation of the sample mean shape}
Whenever one is dealing with an object that is conceptually of infinite, or very high, dimension,
a suitable dimension reduction must inevitably take place to enable computers to handle this object.
Because this process is usually a projection from an infinite dimensional sample space of which the
original object is an element, onto a finite dimensional subspace, we will for convenience refer to it
as a ``projection". In the current situation, the infinite dimensional object is the average of projection operators
$\hat \mu$ in equation \eqref{eq:hat-mu}, which is a positive element in the Hilbert space of
Hilbert-Schmidt operators. Above, this object has been approximated by rather high-dimensional
projection and then successively by projections of lower dimension, in order to arrive at an
approximation of sufficiently low dimension that is still a good representative of the original object.
What constitutes ``good" here has not been established rigorously, but instead primarily on eye ball
fitting, which may, in many cases, work rather well.
A more sophisticated approach seems possible,however, and might be based on a method employed
in the simulation of Brownian motion to determine a suitable number of points at which the values of
the process should be simulated (Gaines, 2012 \cite{Ga:2012}). The objective in this special case was to ensure that
the first few largest eigenvalues of the covariance operator of the projection would approximate those
of the original Brownian motion with prescribed accuracy. This could be achieved by using expansions
for the eigenvalues of the projection in terms of those of the original process, known from perturbation
theory. Since the statistic of main interest in the application considered in this paper is the largest
eigenvalue of $\hat{\Gamma},$ a similar approach should, in principle, be appropriate in the present
context. However, the problem of formally approximating infinite dimensional objects is a topic in its
own right that is beyond the scope of the present paper, and that should, moreover, be considered in
a more general context than presented by the situation at hand. }
\section{Application of the Neighborhood Hypothesis Test for Mean Shape}
{
The test discussed in Section 5 could be performed for a variety of applications.  The most likely
applications involve having a known extrinsic mean shape determined from
historical data.  In such cases, the hypothesis test can be used to determine
whether there is a significant deviation from the historical mean shape.  An
application in agriculture would be determining whether the use of a new fertilizer
treatment results in the extrinsic mean shape of a crop significantly changing from
the historical mean.  Similarly, this test could be performed for quality control
purposes to determine if there is a significant defect in the outline of an produced good.
In practice, $\delta$ will be determined by the application and the decision for a test
would be reached in the standard fashion.  However, for the examples presented here,
there is no natural choice for $\delta$, so one can instead consider setting
$Z=\xi_{1-\alpha}$ and solving for $\delta$ to show what decision would be reached
for any value of $\delta$.  To do so, it is important to understand the role of $\delta$.
The size of the neighborhood around $m_0$ is completely determined by $\delta$.
As such, it follows that smaller values of $\delta$ result in smaller neighborhoods.
In terms of $H_0$, this places a greater restriction on $M_\delta$ and $B_\delta$,
requiring $\mu_E$ to have a smaller distance to $m_0$.
For the examples presented here, the contours are approximated using $k=300$ stopping times,
so the shape space is embedded into $S(300,\mathbb{C})$ to conduct analysis.  In this
environment, consider having two $k$-gons that are identical except for at one time.
If this exceptional point for the second $k$-gon differs from the corresponding
point in the first $k$-gon by a difference of 0.01 units, then the distance between the shapes
inherited from $S(300,\mathbb{C})$ is approximately 0.0141.  For the hypothesis test, if
$\delta=0.0141$, then the neighborhood around $m_0$ would consist of distances between
shapes similar in scope to the situation described above.
}
First, consider an example for which the one sample test for
extrinsic mean shape is performed for sting ray contours. In this case,
the sample extrinsic mean shape for a sample of contours of $n=10$
sting rays  is the shape shown on the left hand side in Fig. \ref{f:one-sample-ray}.
\begin{figure}[!ht]
\begin{center}
\includegraphics[height=4.5cm,width=4.5cm]{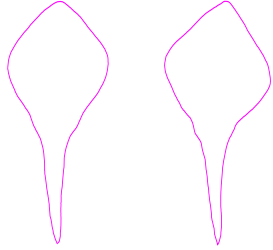}
            \caption{\scriptsize The extrinsic sample mean shape of a sample of
            10 sting ray contours and, respectively, the hypothesized extrinsic mean shape}
             \label{f:one-sample-ray}
\end{center}
\end{figure}
After performing the calculations, it was determined that for an
asymptotic level 0.05 test, the largest value of $\delta$ for which
we would reject the null hypothesis is 0.0290. {  For
perspective, this neighborhood has a radius roughly 2 times larger
than the example with the nearly identical $k$-ads described above.
This means that we would only reject the null hypothesis if we
required the sample extrinsic mean to be nearly identical to the
hypothesized mean. It should also be noted here that the sample size
is small here, but that the conclusion agrees with intuition based
upon a visual inspection of the contours.}
Now consider two examples involving contours of pears.  In this first
case, the sample consists of $n=87$ pears. The sample extrinsic
mean shape and hypothesized extrinsic mean shape are shown in Fig
\ref{f:one-sample-pear1}.
\begin{figure}[!ht]
\begin{center}
\includegraphics[height=4.5cm,width=4.5cm]{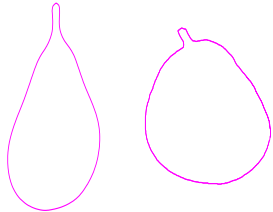}
            \caption{\scriptsize The extrinsic sample mean shape of a sample of
            87 pear contours and, respectively, the hypothesized extrinsic mean shape}
             \label{f:one-sample-pear1}
\end{center}
\end{figure}
It was determined that for an asymptotic level 0.05 test, the
maximum value of $\delta$ for which we would reject the null
hypothesis is 1.2941.  {  This value of $\delta$ is almost 92
times greater than the distance between the nearly identical $k$-ads.
This suggests that even if we greatly relax the constraints for similarity,
the null hypothesis would still be rejected.  This again agrees with intuition.}
In this last example, consider another sample of contours of pears.  In this
scenario, we consider a sample of $n=83$ pears.  The sample extrinsic mean
shape and hypothesized extrinsic mean shape are shown in Fig
\ref{f:one-sample-pear41-45}.
\begin{figure}[!ht]
\begin{center}
\includegraphics[height=4.5cm,width=4.5cm]{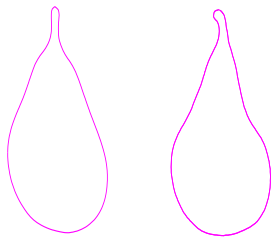}
            \caption{\scriptsize The extrinsic sample mean shape of a sample of
            83 pear contours and, respectively, the hypothesized extrinsic mean shape}
             \label{f:one-sample-pear41-45}
\end{center}
\end{figure}
After performing the calculations, we determined that for an
asymptotic level 0.05 test, the largest value of $\delta$ for which
we would reject the null hypothesis is 0.1969, meaning that our
procedure does not reject the null hypothesis, unless $\delta$ is
smaller then 0.1969. { For perspective, this neighborhood
has a radius nearly 14 times larger than the example with the nearly
identical $k$-ads described above. Unlike in the previous two examples
it is unclear whether the null hypothesis should be rejected in this case
without having a specific application in mind and, as such, this could be
considered a borderline case. }
\section{Bootstrap Confidence Regions for Means of Shapes of Contours}
{  Another method for performing inference, which we
consider now, is through the use of nonparametric nonpivotal bootstrap.
By repeatedly resampling from the available data and computing the distance between
each resampled mean and the sample mean, we can obtain a confidence region
for the extrinsic mean shape (for the sparse case, see Bandulasiri et al. (2008) \cite{BaBhPa:2009} and
Amaral et al (2010) \cite{AmDrPaWo:2010}).
The following examples of 95$\%$ nonparametric nonpivotal bootstrap confidence regions
illustrate this approach using 400 resamples and serve to illustrate a methodology for
visually understanding the confidence regions and their behavior. For each example, the
sample is displayed on the left and the 95$\%$ confidence region is displayed on the right
in blue with the extrinsic sample mean plotted in red.
}
\begin{figure}[!ht]
\centerline{
\includegraphics[height=4.0cm,width=6.66cm]{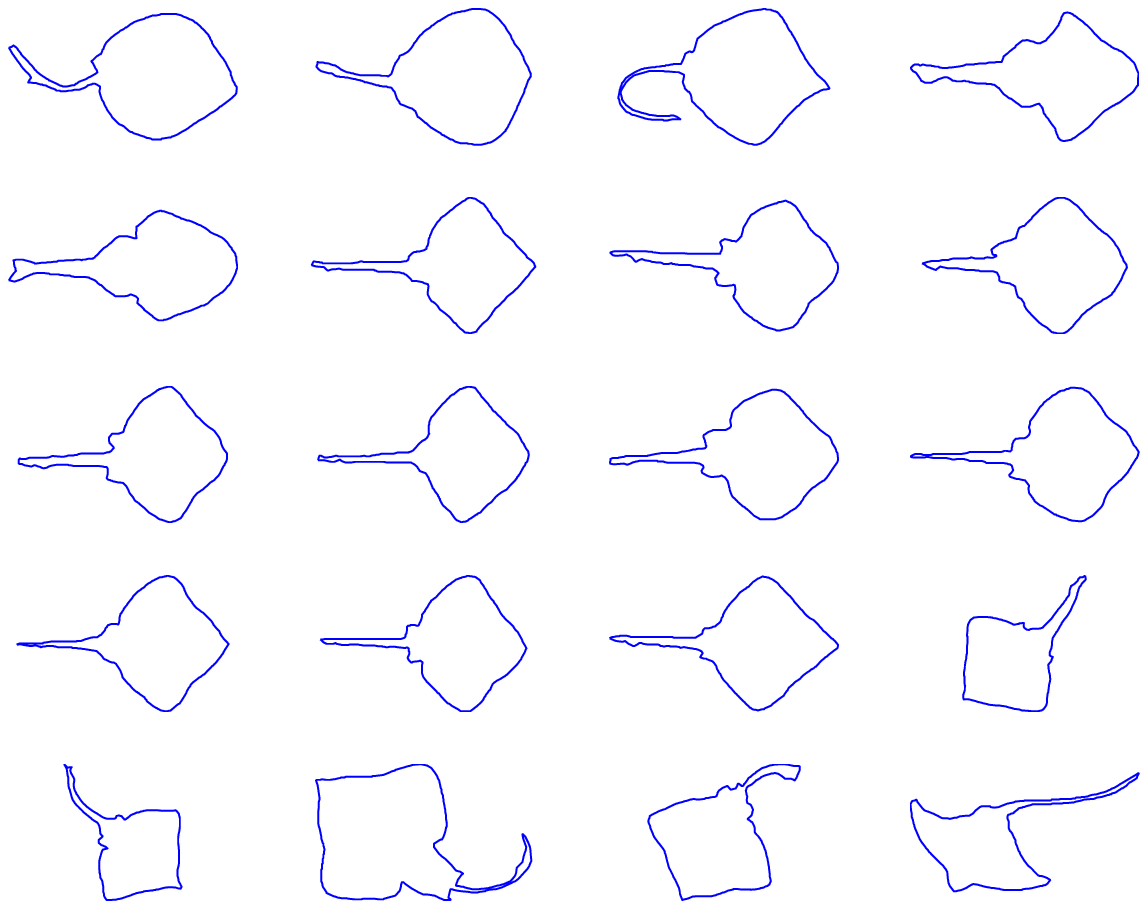}
\includegraphics[height=4.0cm,width=4.0cm]{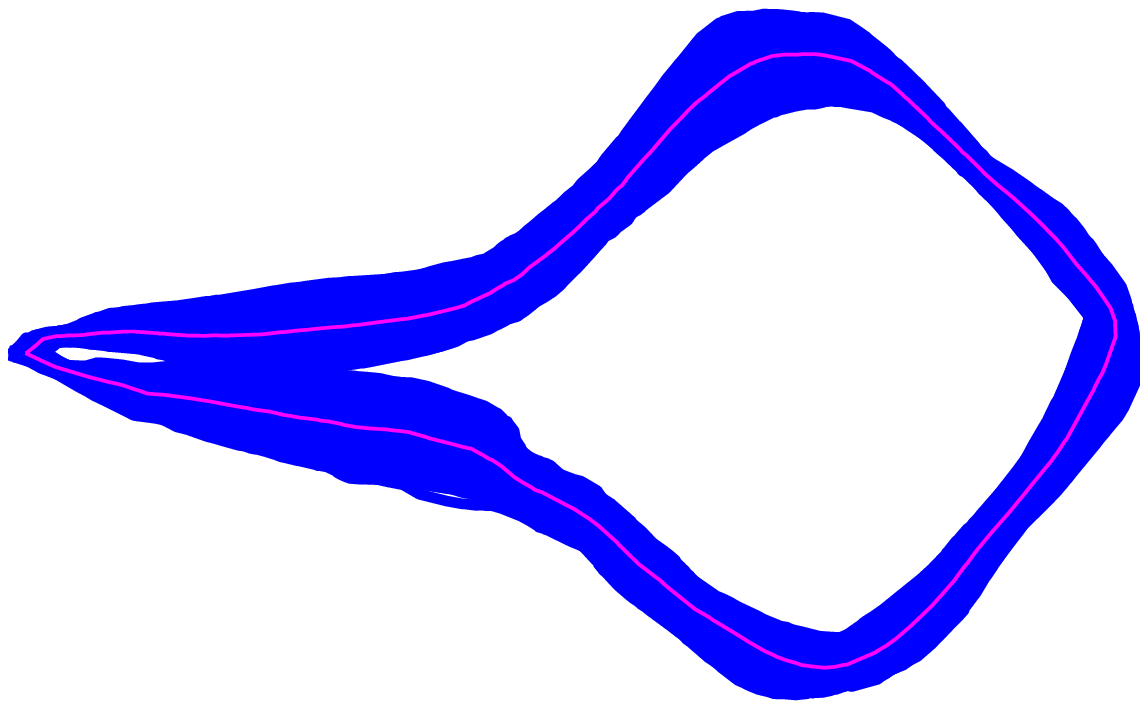}
}
            \caption{\scriptsize Bootstrap 95$\%$ confidence regions for the extrinsic mean shape of stingray contours using a sample of size 20.} \label{f:boot_ray}
\end{figure}
{
The first example, shown in Fig. \ref{f:boot_ray}, reveals that the confidence regions
are wider in the portions of the shape in which there is more variability in the sample.
Here, the bands are thicker in the regions corresponding to the tail and the top
and bottom of the front section of the stingray, where the variability is the greatest.
Secondly, samples with less variability result in narrower confidence regions.  This can
be seen by comparing Figs. \ref{f:boot_ray} and \ref{f:boot_worm}.  It is easy to see that
there is less overall variability in the shapes of the contours of the wormfish than there is
for the stingrays, which is reflected in the widths of the confidence regions.}
Furthermore, the effect of sample size on the confidence regions is clearly displayed in
Fig. \ref{f:boot_pears}. As should be expected, the confidence region constructed using
88 observations is substantially thinner than that constructed using just 20 observations.
\begin{figure}[!ht]
\centerline{
\includegraphics[height=4.0cm,width=6.66cm]{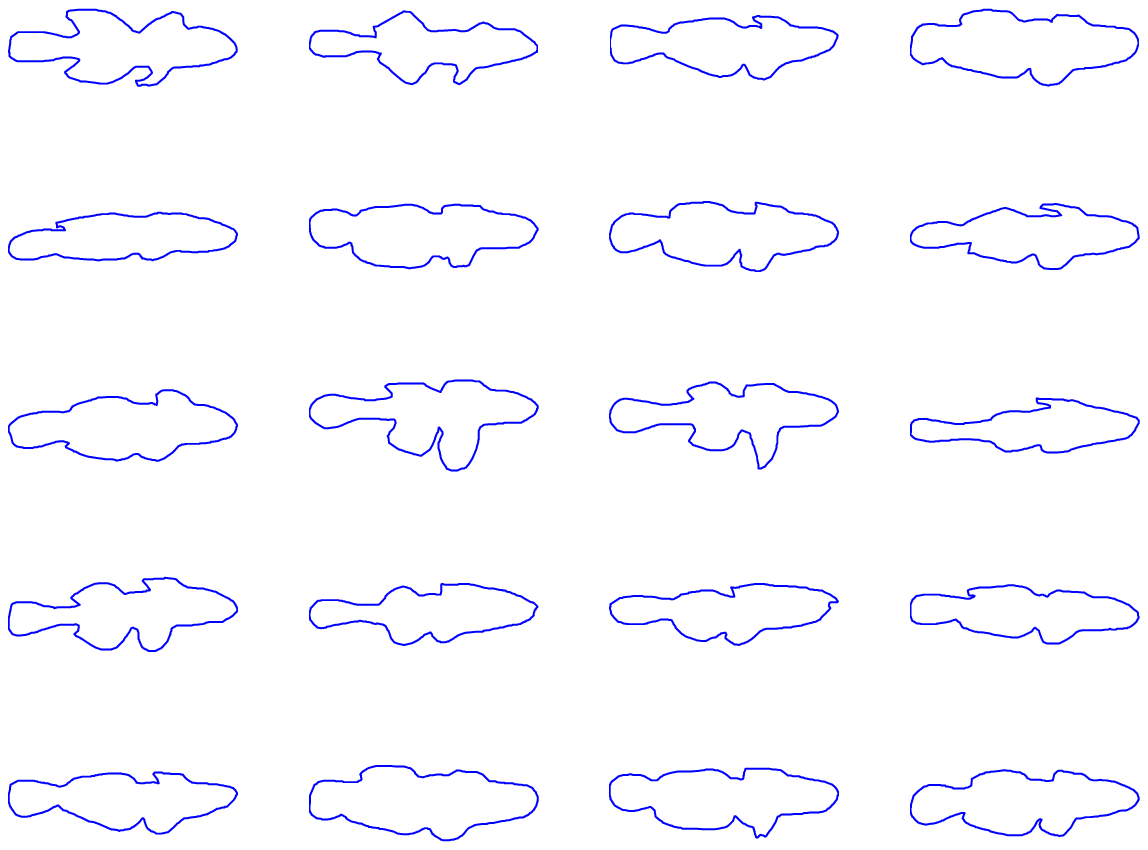}
\includegraphics[height=4.0cm,width=4.0cm]{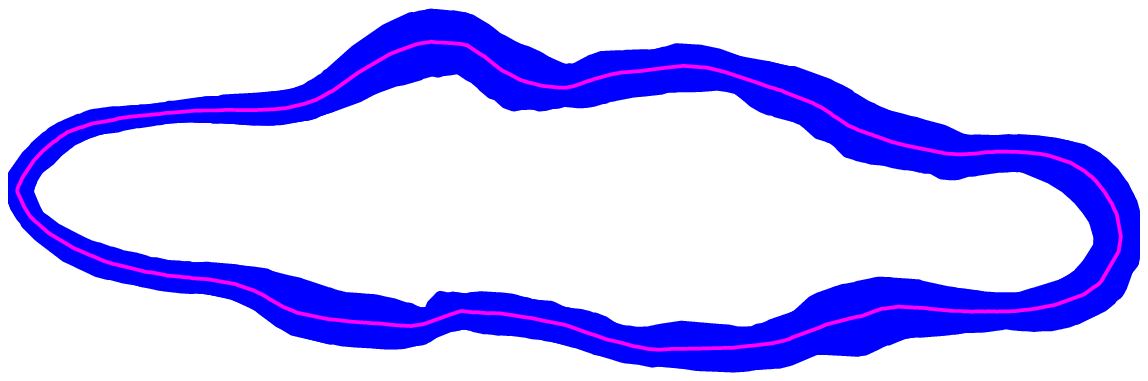}}
            \caption{\scriptsize Bootstrap 95$\%$ confidence regions for the extrinsic mean shape of wormfish contours using a sample of size 20.} \label{f:boot_worm}
\end{figure}
\begin{figure}[!ht]
\mbox{ } \hspace{30pt} {\bf (a)} \hspace{140pt} {\bf (b)} \\
\centerline{
\includegraphics[height=4.0cm,width=4.0cm]{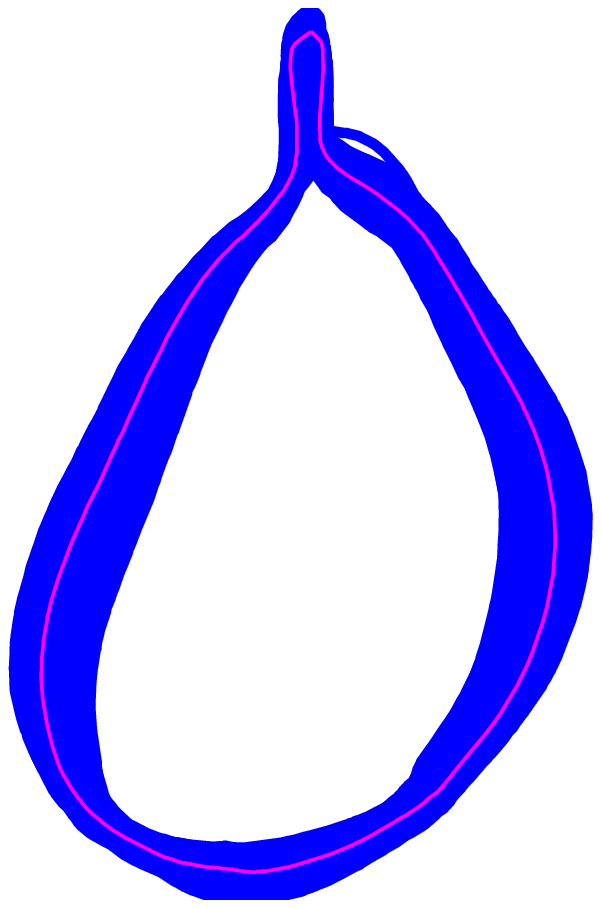}
\includegraphics[height=4.0cm,width=4.0cm]{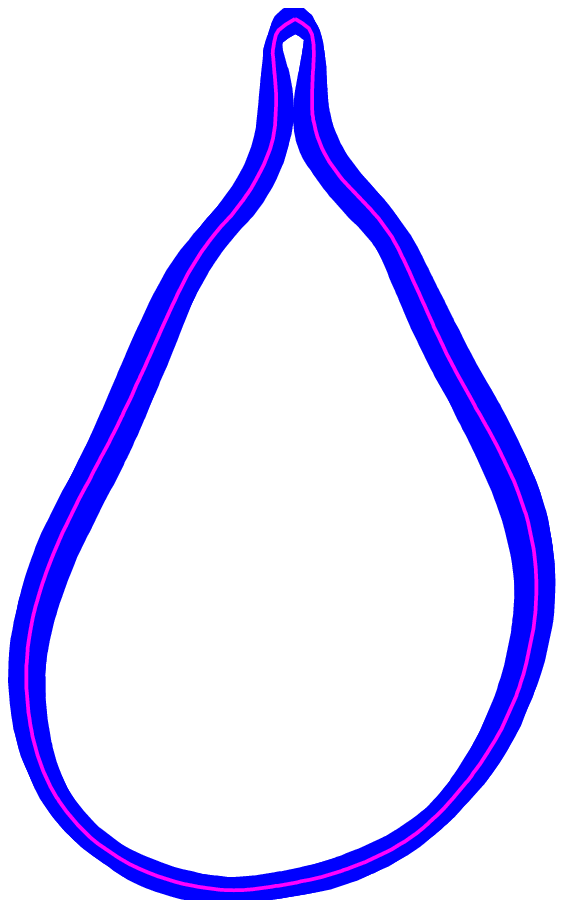}}
            \caption{\scriptsize Bootstrap 95$\%$ confidence region for the extrinsic mean shape for the pears based on (a) 20 observations and (b) 88 observations.} \label{f:boot_pears}
\end{figure}
In addition to being able to obtain sensible and intuitive results, the processing time
needed to compute bootstrap confidence regions for the VW extrinsic mean is small
compared to doing the same using the elastic framework for the analyzing the shape
of planar curves. The higher computational cost is due to a combination of the
intrinsic analysis and the elastic representation.  The calculation of an intrinsic mean
requires the use of an iterative algorithm. {  The square-root elastic
framework of Joshi { et al.} (2007) \cite{JoSrKlJe} adapts the algorithm of Klassen
{ et al.} (2004) for arc-length parametrized curves by inserting a reparametrization
step at each iteration.} This reparametrization step requires the use of either a dynamic
programming algorithm or a gradient descent approach. These time-consuming steps are
repeated a number of times during the calculation of the intrinsic mean, which, when
obtaining a bootstrap confidence region, results in the computational cost being
further compounded.
As an example, this methodology was performed on a sample of hand gestures
representing the letter ``L" using the concepts of elastic shape representation,
as described in Joshi {\it et al.} (2007) \cite{JoSrKlJe}, {  and our methodology.
The resulting confidence regions are given in Fig. \ref{f:elastic_lmark}. Using MATLAB on
a machine running Windows XP on an Intel Core 2 Duo processor running at 2.33 GHz,
these computations required 47.9 hours for the elastic method and only 47.5 seconds
for ours. The difference in the size of the contours displayed in Fig.\ref{f:elastic_lmark}
is due to the approaches using different methods for normalization.}  While we scale the
complex vector denoting the coordinates for the contour at the sampled times to have a
norm of 1, the square-root elastic framework scales curves to have unit length.
While both methods perform well at producing
estimates for mean shape and providing bootstrap confidence regions,
our approach is far more computationally efficient. For a
more detailed account of the advantages of extrinsic analysis of data
on manifolds, especially for obtaining bootstrap confidence regions,
see also Bhattacharya et al.(2012)\cite{BhElLiPaCr:2011}.
\begin{figure}[!ht]
\mbox{ } \hspace{30pt} {\bf (a)} \hspace{140pt} {\bf (b)} \\
\centerline{
\includegraphics[height=4.0cm,width=4.0cm]{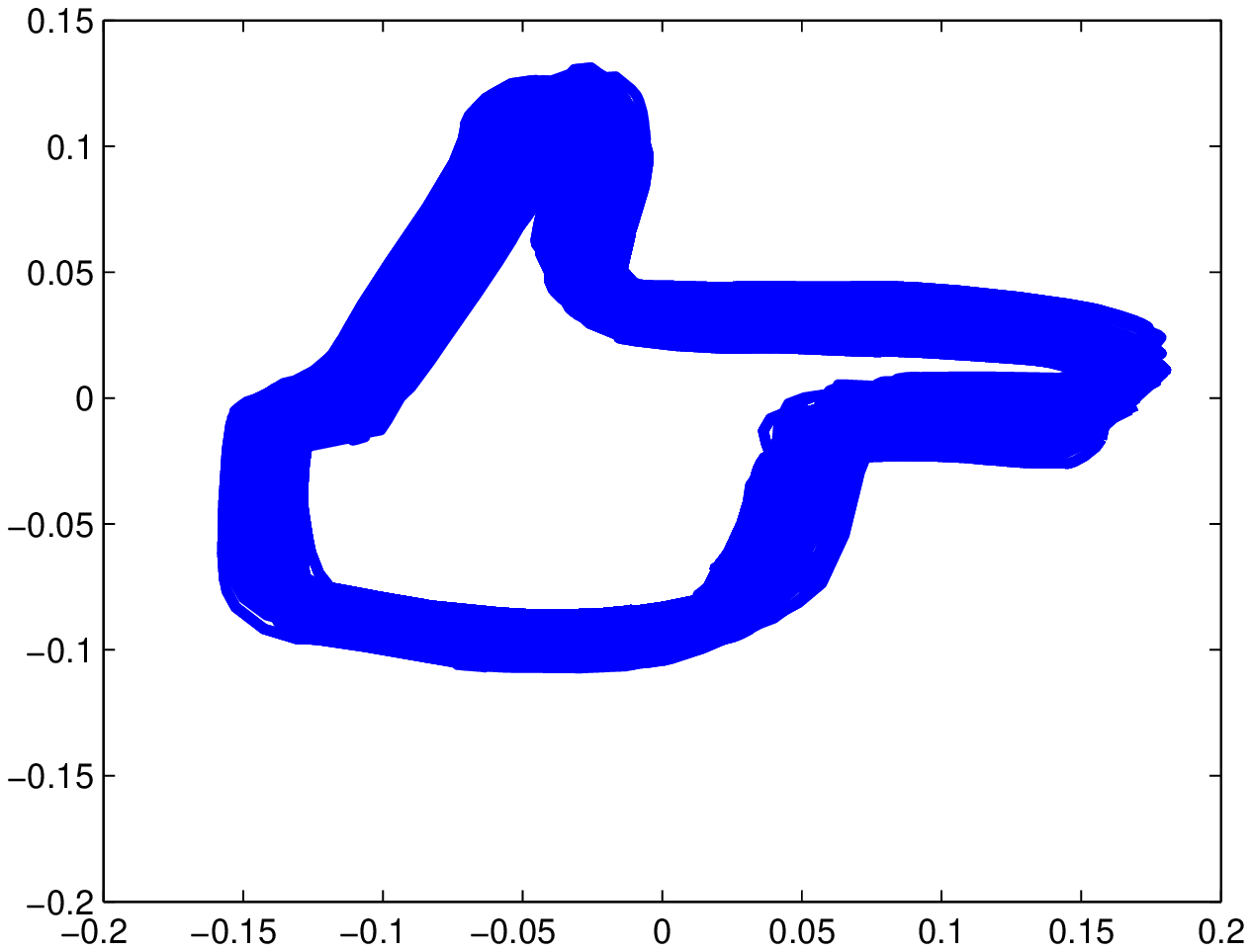}
\includegraphics[height=4.0cm,width=4.0cm]{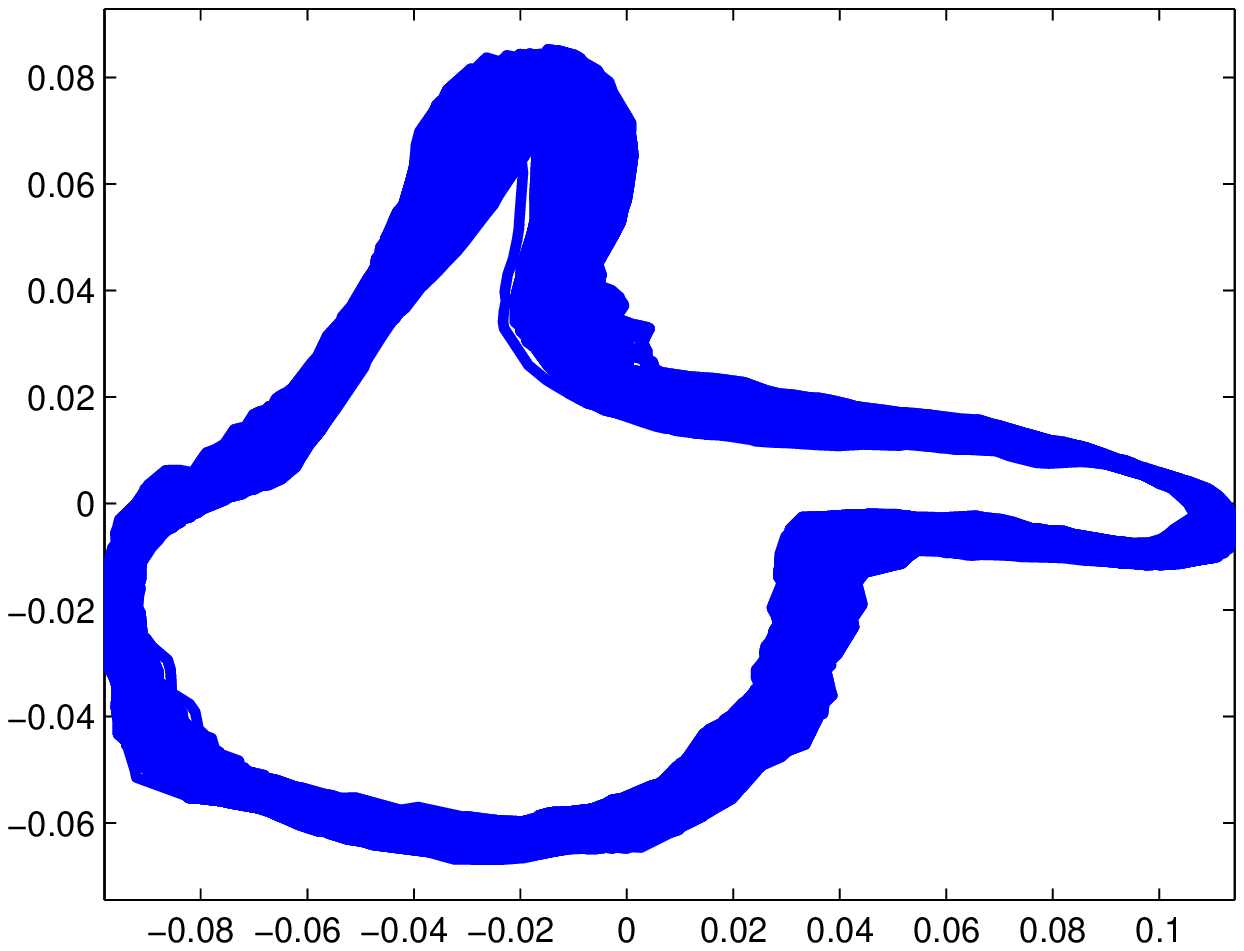}}
            \caption{\scriptsize Bootstrap 95$\%$ confidence regions for (a)
            the intrinsic mean shape, as defined by Joshi et al. (2007),
            and (b) the extrinsic mean shape of the "l" hand gesture.} \label{f:elastic_lmark}
\end{figure}
\section{Discussion}
In this paper, we have described  {  how to address the neighborhood hypothesis for one population mean on a Hilbert manifold. This first paper on data analysis on a Hilbert manifold opens up a new area for data analysis of infinite dimensional objects that can not be represented on a Hilbert space. This is a rich domain for further study, with potential extensions and techniques coming from recent advances  in statistics on finite dimensional manifolds and data analysis on infinite dimensional Hilbert spaces.}
Regarding applications in shape analysis, while our {theory and computational methodology leads
to} the estimation of the extrinsic mean direct similarity shapes of planar {contours}, this
approach could be extended further to any infinite configurations in the Euclidean
plane, including 1-dimensional CW-complexes and planar domains, given that the
plane is separable. For example, one may consider shapes of edge maps obtained
from gray-level images. In these cases, the problem of properly matching becomes
much more difficult because, not only do points on a given edge from one image
need to be matched to corresponding points on the corresponding edge in another
image, but each edge in an image must be matched to the corresponding edge in
another image, as well.
{

Additionally, we have only considered inference techniques for
one-sample problems.  While two-sample and multi-sample
procedures may be more practical for many data analysis
purposes, both the one-sample neighborhood hypothesis test and
the non-pivotal nonparametric bootstrap confidence regions are
useful nonparametric techniques for the estimation of the extrinsic
mean shape for planar contours.  In addition, they serve as an
important first step towards the development and/or adapting of
the desired two-sample and multi-sample procedures.

Furthermore, it should be noted that planar contour data carries
with it innate difficulties that can be challenging to account for
in the data analysis described in this paper.  First, it is
important to maintain a consistent camera position with respect to
the object of interest to ensure that direct similarity shape
analysis is appropriate.  Otherwise, it may be more appropriate to
consider projective shape analysis instead.

Finally, planar contours may commonly depict images of 3D objects.
In the case that the object is relatively flat, such as with  the stingrays
in the example here, a slight shift in the camera angle
may not result in a substantial change in the contour obtained from
the digital image. However, for other 3D objects, such as the dogs
and hand gestures, a slight shift may result in a drastic change in
the form of the associated contour, { which is an issue that has largely
been ignored in the literature.}
In such cases, neither planar direct similarity shape nor planar projective
shape may be an adequate descriptor for analyzing the contour data.
Because of {these issues}, it is important to take great care with planar
shape analysis of contours { that arise from} images of 3D solid {objects.
In such cases, where there is an absence of additional information on the
scenes pictured, this care can help to ensure that a meaningful analysis
can be conducted.}
}
\section*{Acknowledgement}
{ We are grateful to the organizers of the Analysis of Object Data program 2010/2011 at SAMSI, and, in particular, to Hans Georg M$\ddot{u}$ller and to James O. Ramsey for useful conversations on infinite object data analysis regarded as data analysis on Hilbert manifolds. Thanks also to Rabi N. Bhattacharya and John T. Kent for discussions on the subject and to Ben Kimia and Shantanu H. Joshi for providing access to the silhouette data library.}


\begin{thebibliography}{lll}

\bibitem{AmDrPaWo:2010} Amaral, G. J. A. ; Dryden, I. L.; Patrangenaru, V. and
Wood, A.T.A. (2010). Bootstrap confidence regions for the planar
mean shape. {\it Journal of Statistical Planning and Inference}.
{\bf 140}, 3026-3034.
%
\bibitem{Az:1994} Azencott, R. (1994). Deterministic and random deformations ;
applications to shape recognition. {\it Conference at HSSS workshop
in Cortona, Italy.}
%
\bibitem{AzCoYo:1996} Azencott, R.; Coldefy, F. and Younes, L. (1996).
A distance for elastic matching in object recognition Proceedings of
12th ICPR (1996), 687--691.
%
\bibitem{BaBhPa:2009} Bandulasiri, A.; Bhattacharya, R.N. and
Patrangenaru, V. (2009) Nonparametric Inference for Extrinsic Means
on Size-and-(Reflection)-Shape Manifolds with Applications in
Medical Imaging. {\it Journal of Multivariate Analysis}.
\textbf{100} 1867-1882.
%
\bibitem{BaPa:2005} Bandulasiri A. and Patrangenaru, V. (2005). Algorithms for
Nonparametric Inference on Shape Manifolds, {\it Proc. of JSM 2005},
Minneapolis, MN, 1617-1622.
%
\bibitem{BaBa:2008} Bhattacharya A. and Bhattacharya, R. N. (2008). Statistics on Riemannian Manifolds: Asymptotic
Distribution and Curvature. {\it Proceedings of the American
Mathematical Society} 136. 2957-2967.
%
{  \bibitem{BhElLiPaCr:2011} Bhattacharya, R.N;
Ellingson, L; Liu, X; Patrangenaru, V; and Crane, M. (2012)
Extrinsic Analysis on Manifolds is Computationally Faster than
Intrinsic Analysis, with Application to Quality Control by Machine
Vision. To appear in {\em Applied Stochastic Models in Business and
Industry.}}
%
\bibitem{BhPa:2003} Bhattacharya, R.N. and Patrangenaru, V. (2003). Large sample theory
of intrinsic and extrinsic sample means on manifolds-Part I,{\it
Ann. Statist.} {\bf 31}, no. 1, 1-29.
%
\bibitem{BhPa:2005} Bhattacharya, R.N. and Patrangenaru, V. (2005). Large sample
theory  of intrinsic and extrinsic sample means  on manifolds- Part
II, {\it Ann. Statist.}, {\bf 33}, No. 3, 1211- 1245.
%
\bibitem{DaPoRo:1982} Dauxois, J., Pousse, A. and Romain, Y. (1982).
Asymptotic theory for the principal component analysis of a vector random function: some applications to statistical inference. {\em J. Multivariate Anal.} {\em 12}, 136--154.
%
\bibitem{Efron:79} Efron, B. (1979) Bootstrap methods: another look at the jackknife.
{it Ann. Statist. } {\bf 7}, No. 1, 1--26.
%
\bibitem{Fr:1948} Fr\'echet, M. (1948). Les \'elements
 al\'eatoires de nature
quelconque dans un espace distanci\'e. {\it Ann. Inst. H.
Poincar\'e} {\bf 10}, 215--310.
%
\bibitem{Ga:2012} Gaines, G. (2012). Random perturbation of a self-adjoint operator with a multiple eigenvalue.  Dissertation. Department of Mathematics and Statistics, Texas Tech University.
%
\bibitem{Gr:1993} Grenander, U. (1993). {\it General Pattern Theory.}
Oxford Univ. Press.
%
\bibitem{Go:1963} Goodman, N. R. (1963). Statistical analysis based on a certain multivariate complex Gaussian distribution.
(An introduction)  {\it Ann. Math. Statist.}  {\bf 34}, 152--177.
%
\bibitem{Hall:1997} Hall, P. (1992). {\it The bootstrap and Edgeworth
expansion}. Springer Series in Statistics, New York.
%
\bibitem{haMuWa:2006} Hall, P., M$\ddot{u}$ller, H.-G., Wang, J.-L. (2006). Properties of principal component methods for functional and longitudinal data analysis. {\em Ann. Statist.} {\bf 34}, 1493--1517.
%
\bibitem{HuZi:2006} Huckemann, S. and Ziezold, H. (2006). Principal component analysis for Riemannian manifolds, with an application to
triangular shape spaces. {\it  Adv. in Appl. Probab.}  {\bf 38},
no. 2, 299--319.
%
\bibitem{HuHo:2009} Huckemann, S. and Hotz, T. (2009).
Principal component geodesics for planar shape spaces. {\it J.
Multivariate Anal.}  {\bf 100},  no. 4, 699--714.
%
\bibitem{JoSrKlJe} Joshi S.; Srivastava, A. ; Klassen, E. and Jermyn, I.
(2007). Removing Shape-Preserving Transformations in Square-Root
Elastic (SRE) Framework for Shape Analysis of Curves. {\it Workshop
on Energy Minimization Methods in CVPR (EMMCVPR)}. August.
%
\bibitem{KaSr:2007}  Kaziska, D. and Srivastava, A. (2007). Gait-Based Human Recognition
by Classification of Cyclostationary Processes on Nonlinear Shape
Manifolds, JASA, 102(480): 1114-1124.
%
\bibitem{KendBaCaL:1999} Kendall, D.G.; Barden, D.; Carne, T.K. and Le, H. (1999).
{\it Shape and Shape Theory.} Wiley, New York.
%
\bibitem{Kend:1984} Kendall, D.G. (1984), Shape manifolds, Procrustean metrics,
and complex projective  spaces. {\it Bull.\ London Math.\ Soc.} {\bf
16}  81-121.
%
\bibitem{Ke:1992} Kent, J.T. (1992), New directions in shape analysis. {\it
The Art of Statistical Science, A Tribute to G.S. Watson}, 115--127.
Wiley Ser. Probab. Math. Statist. Probab. Math. Statist., Wiley,
Chichester, 1992.
%
{  \bibitem{Kimia} Kimia, Ben. A Large Binary Image
Database. http://www.lems.brown.edu/$\sim$ dmc/
%
\bibitem{KlSrMiJo:2004} Klassen, E.; Srivastava, A. ; Mio, W. and
Joshi, S. H. (2004). Analysis of Planar Shapes Using Geodesic Paths
on Shape Spaces, {IEEE Transactions on Pattern Analysis and Machine
Intelligence} {\bf 26} 372 - 383.}
%
\bibitem{KuLe:2003} Kume, A. and Le, H. (2003). On Fr\'echet means in simplex shape spaces.
{\it Adv. in Appl. Probab.} {\bf  35 }, 885--897.
%
\bibitem{KuLe:2000} Kume, A. and Le, H. (2000). Estimating Fr\'echet means in Bookstein's
shape space.  {\it Adv. in Appl. Probab.} {\bf  32 }, 663--674.
%
\bibitem{Lang:1986} Lang, S. (1986) {\em
Differential manifolds. } Springer, New York.
%
\bibitem{Le:2001} Le, H. (2001). Locating Fr\'echet Means with
Application to Shape Spaces . {\it Advances in Applied Probability},
{\bf 33}. 324--338.
%
\bibitem{LeKu:2000}Le, H. and Kume, A. (2000). The Fr\'echet Mean Shape and
the Shape of the Means. {\it Advances in Applied Probability}, {\bf
32}. 101-113
%
\bibitem{LeKu:2000} Le, H. and Kume, A.(2000). The Fr\'echet mean shape and the
shape of the means. {\it Adv. in Appl. Probab.} {\bf  32 },
101--113.
%
\bibitem{Lo:1977} Lo$\grave{e}$ve, M. (1977). {\em Probability Theory}(fourth ed.) Springer-Verlag, Berlin.
\bibitem{MicMum:2004} Michor, P. W. and Mumford, D. (2004). Riemannian Geometries
on Spaces
of Plane Curves. {\it Journal of the European Mathematical Society},
{\bf 8}, 1 - 48.
%
\bibitem{MaPa:01} Mardia, K.V. and Patrangenaru, V. (2001) On affine and projective
shape data analysis. In ``Functional and Spatial Data
Analysis''. {\em Proceedings of the 20th LASR Workshop,
edited by K.V. Mardia\& R.G. Aykroyd }, Leeds University Press, Leeds, 39--45.
%
\bibitem{MiSr:2004} Mio, W. and Srivastava, A. (2004). Elastic-String Models
for Representation and Analysis of Planar Shapes. {\it Proceedings
of the IEEE Computer Society International Conference on CVPR}.10--15.
%
\bibitem{MiSrKl:2005} Mio, W.; Srivastava, A. and Klassen, E. (2004)Interpolation by Elastica in Euclidean Spaces.{\it Quarterly ofApplied Math.} {\bf 62}, 359 - 378 .
%
\bibitem{MiSrJo:2005} Mio, W.; Srivastava, A. and Joshi, S. (2007). On the
Shape of Plane Elastic Curves. {\it International Journal of Computer Vision}, {\bf 73}, 307--324.
%
\bibitem{Mueler:2006} M$\ddot{u}$ller, H.-G., Stadtm$\ddot{u}$ller, U. and Yao, F. (2006). Functional variance processes. {\em J. Amer. Statist. Assoc.} {\bf 101} 1007--1018.
%
\bibitem{Munk:98} Munk, A. and Dette, H. (1998) Nonparametric comparison of
several regression functions: exact and asymptotic theory.
\textit{Ann. Statist.} \textbf{26}, 2339-2368.
%
\bibitem{MuPaPaPaRu:2008} Munk, A.; Paige, R.; Pang, J. ; Patrangenaru, V. and Ruymgaart, F. H.(2008).
The One and Multisample Problem for Functional Data with
Applications to Projective Shape Analysis. {\it J. of Multivariate
Anal.} {\bf. 99}, 815-833.
%
\bibitem{Pa:1998} Patrangenaru, V. (1998). Asymptotic Statistics on Manifolds,
PhD Dissertation {\it Indiana University}.
%
\bibitem{PaLiSu} Patrangenaru, V., Liu, X. and Sugathadasa, S. (2010).
Nonparametric 3D Projective Shape Estimation from Pairs of 2D Images
- I, In Memory of W.P. Dayawansa. {\it Journal of Multivariate
Analysis}. {\bf 101}, 11-31.
%
\bibitem{RaSi:2005}  Ramsay J. O. and Silverman, B. W. ( 2005). {\em Functional Data Analysis}, Springer Series in Statistics, Springer, 2nd edition.
%
\bibitem{SeKlKi:2003} Sebastian, T.B. ; Klein, P.N. and Kimia,B.B. (2003).
On Aligning Curves. {\it IEEE Trans. Pattern Analysis and Machine
Intelligence} {\bf 25}, no. 1, 116--125.
%
\bibitem{Small:1996} Small, C. G. (1996). {\it The Statistical Theory of Shape},
 Springer-Verlag, New York.
%
\bibitem{SrKl:2002} Srivastava, A. and Klassen, P.E. (2002). Monte Carlo extrinsic
estimators
for manifold-valued parameters{\it IEEE Transactions on Signal
Processing.} {\bf 50}, 299-308.
%
\bibitem{SrJoMiLi:2005} Srivastava, A. ; Joshi, S. Mio, W. and Liu, X. (2005).
Statistical Shape
Analysis: Clustering, Learning and Testing. {\it IEEE Trans. Pattern
Analysis and Machine Intelligence}, {\bf 27}. 590--602.
%
\bibitem{Yo:1998} Younes, L. (1998). Computable elastic distance between shapes. {\it SIAM
Journal of Applied Mathematics}, {\bf 58}, 565-586.
%
\bibitem{Yo:1999} Younes, L. (1999).
Optimal matching between shapes via elastic deformations. {\it
Journal of Image and Vision Computing}, {\bf 17}, 381-389.

\bibitem{YoMiShMu:2008} Younes, L.; Michor, P. W.; Shah, J. and
Mumford, D. (2008). A metric on shape space with explicit geodesics.
{\it Rend. Lincei Mat. Appl.} {\bf 19} 25-57.
%
{  \bibitem{ZhRo:1972} Zahn, C. T. and Roskies. (1972).
Fourier Descriptors for Plane Closed Curves. {\em IEEE Tras.
Computers} {\bf 21} 269-281.}
%
\bibitem{Zi:2000} Ziezold, H. (1998). Some aspects of random shapes.
(English summary). {\em Numbers, information and complexity}
(Bielefeld, 1998), 517--523, Kluwer Acad. Publ., Boston, MA.
%
\bibitem{Zi:1994} Ziezold, H. (1994). Mean figures and mean shapes applied
to biological figure and shape distributions in the plane. {\it
Biometrical J. } {\bf 36},  no. 4, 491--510.
%
\bibitem{Zi:1977} Ziezold, H. (1977). On expected figures and a strong law
of large numbers for random elements in quasi-metric spaces. {\it
Transactions of the Seventh Prague Conference on  Information
Theory, Statistical Decision Functions, Random Processes and of the
Eighth European Meeting of Statisticians}, Vol.\ {\bf A}, 591-602.
Reidel, Dordrecht.
\end{thebibliography}
\end{document}